\documentclass[11pt,letterpaper]{amsart}
\usepackage[margin=1.25in]{geometry}

\usepackage{amsfonts}
\usepackage{mathtools}
\usepackage{amsmath, amsthm, amssymb}
\usepackage{thmtools, thm-restate}

\usepackage{enumitem}
\usepackage{caption}
\usepackage{subcaption}
\usepackage[only,llbracket,rrbracket]{stmaryrd}
\usepackage[svgnames]{xcolor}
\usepackage[colorlinks,citecolor=DodgerBlue,linkcolor=Indigo]{hyperref}
\usepackage[nameinlink]{cleveref}
\usepackage{relsize}
\usepackage{crossreftools}

\usepackage{xparse}
\usepackage{xspace}
\usepackage{xstring}
\makeatletter
\newcommand{\sg}[1]{%
{
    \exploregroups
    \IfSubStr{#1}{_}{%
        \StrBefore{#1}{_}[\@tempa]%
        \StrBehind{#1}{_}[\@tempb]%
        \sg{\@tempa}_\@tempb%
    }{%
        \IfSubStr{#1}{^}{%
            \StrBefore{#1}{^}[\@tempc]%
            \StrBehind{#1}{^}[\@tempd]%
            \sg{\@tempc}^\@tempd%
        }{%
            \IfSubStr{#1}{'}{%
                \StrBefore{#1}{'}[\@tempe]%
                \StrBehind{#1}{'}[\@tempf]%
                \sg{\@tempe}'\@tempf
            }{%
                \widehat{#1}
            }%
        }%
    }%
}
}%
\makeatother

\declaretheorem[numberwithin=section]{lemma}

\declaretheorem[sharenumber=lemma,
refname={prop.,props.},
Refname={Proposition,Propositions}
]{proposition}
\declaretheorem[sharenumber=lemma,
refname={cor.,cors.},
Refname={Corollary,Corollaries}
]{corollary}
\declaretheorem[style=definition,sharenumber=lemma,
refname={def.,defs.},
Refname={Definition,Definitions}
]{definition}
\declaretheorem[style=definition,sharenumber=lemma,
refname={convention,conventions},
Refname={Convention,Conventions}
]{convention}
\declaretheorem[style=remark,sharenumber=lemma,
refname={rmk.,rmks.},
Refname={Remark,Remarks}
]{remark}
\declaretheorem[style=definition,numbered=no,
title={Running Example}
]{example}

\crefname{section}{sec.}{sections}
\Crefname{section}{Section}{Sections}

\newlist{enumdef}{enumerate}{1} 
\setlist[enumdef]{label=\upshape(\arabic*), ref=\upshape\thedefinition(\arabic*)}
\crefname{enumdefi}{def.}{defs.}
\Crefname{enumdefi}{Definition}{Definitions}

\newlist{enumprop}{enumerate}{1} 
\setlist[enumprop]{label=\upshape(\arabic*), ref=\upshape\theproposition(\arabic*)}
\crefname{enumpropi}{prop.}{props.}
\Crefname{enumpropi}{Proposition}{Propositions}

\newcommand{\refitem}[1]{(\crtcrefcountervalue{#1})}

\newcommand{\ZZ}{\mathbb{Z}}
\newcommand{\RR}{\mathbb{R}}
\newcommand{\FF}{\mathbb{F}}
\newcommand{\ep}{\varepsilon}
\newcommand{\halfline}{[0, \infty)}
\newcommand{\skele}[2][0]{#2^{(#1)}}
\newcommand{\zskele}[1]{\skele[0]{#1}}
\newcommand{\cells}[2][0]{#2^{[#1]}}
\newcommand{\verts}[1]{\cells[0]{#1}}
\newcommand{\edges}[1]{\cells[1]{#1}}
\newcommand{\faces}[1]{\cells[2]{#1}}
\newcommand{\ncells}[2][n]{\mu_{#1}(#2)}
\newcommand{\nverts}[1]{\ncells[0]{#1}}
\newcommand{\nedges}[1]{\ncells[1]{#1}}

\newcommand{\subeq}{\subseteq}
\newcommand{\subneq}{\subsetneq}
\newcommand{\supeq}{\supseteq}
\newcommand{\term}[1]{\textit{\textbf{\boldmath #1}}}
\newcommand{\cl}[1]{\overline{#1}}
\newcommand{\from}{\colon}
\newcommand{\bd}{\partial}
\newcommand{\Ends}{\mathcal{E}}
\newcommand{\seq}[2][]{\set{#2}_{#1}}
\newcommand{\pmn}{{\pm n}}
\newcommand{\restr}[2]{#1|_{#2}}
\newcommand{\inv}{{-1}}
\newcommand{\ginv}{{g^{-1}}}
\newcommand{\gsinv}{{g_*^{-1}}}
\newcommand{\psigoo}{\psi_{g}^{\infty}}
\newcommand{\psigoopm}{\psi_{g^{\pm 1}}^{\infty}}
\newcommand{\homeo}{\approx}
\newcommand{\homot}{\simeq}

\newcommand{\moo}{{-\infty}}
\newcommand{\pmoo}{{\pm \infty}}

\newcommand{\pt}{{\mathord{*}}}
\newcommand{\trunc}[1]{{#1}^{\flat}}
\newcommand{\wt}{\widetilde}
\newcommand{\Gam}{\Gamma}
\newcommand{\gX}{\Xi}
\newcommand{\gY}{\Upsilon}
\newcommand{\id}{\mathrm{id}}
\newcommand{\FBC}[2][{}]{{#2 \rtimes_{#1} \ZZ}}
\NewDocumentCommand{\map}{o m g}{%
    \IfNoValueTF{#1}
    {#2 \to \IfNoValueTF{#3}{#2}{#3}}
    {#1 \from #2 \to \IfNoValueTF{#3}{#2}{#3}}%
}
\newcommand{\cU}{\mathcal{U}}

\DeclarePairedDelimiter{\set}{\lbrace}{\rbrace}
\DeclarePairedDelimiter{\abs}{\lvert}{\rvert}
\DeclarePairedDelimiter{\gen}{\langle}{\rangle}
\let\path\relax
\DeclarePairedDelimiter{\path}{\llbracket}{\rrbracket}

\DeclareMathOperator{\interior}{int}

\DeclareMathOperator{\Maps}{Maps}
\DeclareMathOperator{\fr}{fr}

\DeclareMathOperator{\sgn}{sgn}
\DeclareMathOperator{\St}{St}
\DeclareMathOperator*{\bigast}{\scalebox{1.5}{\raisebox{-0.2ex}{$\ast$}}}

\newcommand{\pd}{\abs}

\begin{document}
\title{Embeddings of mapping tori for end-periodic graph maps}

\author[A. R. Smith]{Adam R. Smith}
\address{Department of Mathematics\\ Temple University}
\email{\href{mailto:smith.adam@temple.edu}{smith.adam@temple.edu}}

\begin{abstract}
    End-periodic homotopy equivalences of infinite, locally finite graphs serve as dimension-one analogs of the end-periodic automorphisms traditionally defined on infinite-type surfaces.
    We demonstrate that if $\Gamma$ is an infinite graph with finitely many ends, and $g \colon \Gamma \to \Gamma$ is end-periodic, then its mapping torus $Z_g$ admits a flowline-preserving homotopy equivalence with a finite 2-complex.
    With additional hypotheses on $g$, this compactified mapping torus subsequently embeds in the mapping torus of a homotopy equivalence on a finite-rank graph via a $\pi_1$-injective, flow-preserving map.
    We prove that every mapping class of $\Gamma$ arising from an end-periodic homotopy equivalence contains a representative whose mapping torus realizes such an embedding.
\end{abstract}

\maketitle

\setcounter{tocdepth}{2}
\tableofcontents

\section{Introduction}
End-periodicity is a dynamical behavior arising naturally from a variety of contexts within low dimensional topology. 
Given an infinite-type surface $S$, a homeomorphism $\map[f]{S}$ is said to be \emph{end-periodic} if it has a positive power $f^p$ under which every end of $S$ is either attracting or repelling. 
Said another way, $f$ is end-periodic if $f^p$ behaves like a nontrivial translation everywhere outside some compact subsurface.

Originating in unpublished work of Handel--Miller, the investigation of end-periodic automorphisms of infinite-type surfaces is an ongoing effort, 
with an extensive body of literature already devoted to the subject (we refer the reader to \cite{CantwellConlonFenley} for a comprehensive treatment of the theory).
The mapping tori of these objects are of particular interest, and have deep connections to the geometry and dynamics of foliated 3-manifolds.
An overview of the notable properties of end-periodic mapping tori can be found in \cite{Loving}; see also \cite{FieldKentLeiningerLoving}, \cite{FieldKimLeiningerLoving}, \cite{Whitfield}, and \cite{LandryMinskyTaylor} for various perspectives.

While the term \emph{end-periodic} has traditionally been reserved for automorphisms of infinite-type surfaces, 
the phenomenon of end-periodicity is by no means exclusive to this context.
Considering the long and successful tradition of analogizing surface behavior with patterns observed one dimension lower, 
it is perhaps unsurprising that a similar notion for maps on infinite graphs turns out to be quite interesting in its own right.
In this spirit, a self map $g$ of an infinite, locally finite graph $\Gamma$ will be deemed end-periodic (\cref{def:end-periodic}) if it restricts to a homeomorphism outside some finite core subgraph, and has a power under which each end of $\Gamma$ is either attracting or repelling.
Such conditions are intended to make end-periodic homotopy equivalences of graphs the analogs of end-periodic homeomorphisms of surfaces.

Compared to the several decades' development underlying the theory of end-periodic surface automorphisms, the study of end-periodic graph maps is still young.
In her thesis \cite{MeadowMacLeod}, Meadow-MacLeod introduced the definition and observed such maps naturally arise within mapping tori of finite graph maps.
He--Wu, adopting the same definition, show the existence of relative train track representatives for end-periodic homotopy equivalences \cite{HeWu}.
The present article contributes to this emerging body of work.
In it, we embark on the first comprehensive investigation of mapping tori for end-periodic graph maps.

\subsection{Main findings}
Our first result is a fundamental observation: the mapping torus of an end-periodic map admits a flow-preserving compactification.
\begin{restatable*}{theorem}{compactifiedMappingTorusTheorem}
    \label{thm:compactified-mapping-torus}
    Let $\Gamma$ be an infinite connected graph with finitely many ends, and \(\map[g]{\Gamma}\) an end-periodic map.
    The mapping torus $Z_g$ is homotopy equivalent to a finite 2-complex \(W_g\) via a flow-preserving embedding.
    In particular, \(W_g\) contains a pair of 1-subcomplexes \(\bd_+ W_g\) and \(\bd_- W_g\), 
    disjoint from one another and oppositely cooriented relative to the semiflow,
    such that $Z_g$ is homeomorphic to $W_g - (\bd_+ W_g \cup \bd_- W_g)$.
\end{restatable*}

The space $W_g$ constructed in \cref{thm:compactified-mapping-torus}, which we call the \emph{compactified mapping torus} of $g$, 
has precedent in the world of end-periodic surface homeomorphisms, where it is well-known that if $f$ is an end-periodic homeomorphism of an infinite-type surface, 
then its mapping torus admits a similar flowline-preserving compactification, and can be realized as the interior of a compact 3-manifold $N_f$ whose boundary is a finite union of closed surfaces \cite{Fenley97}.
Our analysis shows that certain features of the boundary of $N_f$ carry over to $W_g$ when $g$ is a homotopy equivalence;
for instance, Euler characteristic is conserved between $\bd_+ W_g$ and $\bd_- W_g$, 
and the inclusion of each boundary component is necessarily $\pi_1$-injective (see \cref{cor:euler-characteristic-preserved,lemma:boundary-inclusion-pi1-injective}).
While these mirror established facts about mapping tori of end-periodic surface homeomorphisms, the way we prove them is often quite different.
Notably, in the absence of a manifold structure on $W_g$, the most effective techniques tend to be of a combinatorial nature. 

Another important distinction between $Z_g$ and $Z_f$ has to do with their dynamic structures.
While the latter space is equipped with a global suspension flow as the mapping torus of a homeomorphism, 
the lack of a general inverse for $g$ means $Z_g$ can only be furnished with a semiflow.
This is a common challenge when moving from surfaces to graphs;
even if we assume $g$ is a homotopy equivalence, a homotopy inverse $g'$ is not unique, and may fail to posses the same topological properties as the original.
For example, even if $g$ is a proper map, it is not considered a \emph{proper homotopy equivalence} unless $g'$ can also be made proper, 
with $g g'$ and $g' g$ both properly homotopic to the identity.
While this condition is nontrivial in general (see \cite[Example 4.1]{AlgomKfirBestvina}), 
the assumption that $g$ is end-periodic gives more control over the structure of a homotopy inverse.
In particular, we show the following.

\begin{restatable*}{theorem}{homotopyInverseTheorem}
    \label{thm:homotopy-inverse}
    Let $\Gamma$ be an infinite connected graph. 
    For any end-periodic homotopy equivalence $\map[g]{\Gamma}$, there exists a homotopy inverse $\map[g']{\Gamma}$ for $g$ which is end-periodic, and restricts to $g^{-1}$ on a neighborhood of each end.
\end{restatable*}

Since end-periodic maps are proper by definition, it follows that end-periodic homotopy equivalences are automatically PHEs. 
Consequently, the proper homotopy class of every end-periodic homotopy equivalence $\map[g]{\Gamma}$ is an element of the mapping class group $\Maps(\Gamma)$ of $\Gamma$, 
defined as in \cite{AlgomKfirBestvina} to consist of proper homotopy equivalences of $\Gamma$ up to proper homotopy.
This gives us a notion of end-periodicity for mapping classes: we say $\zeta \in \Maps(\Gamma)$ is end-periodic if it has an end-periodic representative.
Our final result concerns embeddings of mapping tori for end-periodic homotopy equivalences:

\begin{restatable*}{theorem}{embeddingTheorem}
  \label{thm:embedding-thm}
  Suppose $\zeta$ is an end-periodic mapping class of an infinite connected graph with finitely many ends.
  There exists an end-periodic representative $g \in \zeta$, along with a finite graph $\Theta$ and homotopy equivalence $\map[f]{\Theta}$, 
  such that the mapping torus $Z_g$ embeds in $Z_f$ via a $\pi_1$-injective, flow-preserving map.
\end{restatable*}

The crux of the proof of \cref{thm:embedding-thm} is a construction which we call \emph{coupling}, 
the inspiration for which comes from the doubled mapping tori of Landry--Minsky--Taylor \cite[Section 3]{LandryMinskyTaylor}.
Any two compactified mapping tori $W_g$ and $W_{g'}$ whose boundaries are related by a coorientation-reversing homeomorphism 
$\map[h]{\bd_\pm W_g}{\bd_\mp W_{g'}}$ can be attached to form a space $M \coloneq W_g \sqcup_h W_{g'}$, which we call their \emph{$h$-couple}.
This is a finite 2-complex with a suspension semiflow $\psi$ descended from the natural semiflows on the original mapping tori, 
and in the case where $g$ and $g'$ are homotopy equivalences, we show the inclusion $W_g \to M$ must be $\pi_1$-injective.
With additional assumptions on the gluing map $h$, one can produce a finite graph $\Theta$ embedded in $M$, 
transverse to $\psi$, and such that its first return map $\map[f]{\Theta}$ is a homotopy equivalence;
we say that $g$ and $g'$ are \emph{compatible} if the boundaries of their compactified mapping tori can be joined by such an $h$.
Our construction ensures $M$ is homeomorphic to $Z_f$, 
so the result follows once we show every end-periodic mapping class has a representative that belongs to a compatible end-periodic pair.

\subsection{Consequences and further questions}
As is often the case when analogizing the behavior of graphs and surfaces, these topological findings have algebraic implications.
Indeed, one of the original motivations for this work was group theoretic conjecture of Chong--Wise, 
recently proved by Linton in \cite{Linton:EmbeddingFBCGroups}, 
which posits that every finitely generated free-by-cyclic group arises as a subgroup of some (fg free)-by-cyclic group \cite[Conjecture 1.2]{ChongWise}.
End-periodic mapping classes represent a natural starting point when considering such objects:
if $\Gamma$ is an infinite graph, and $\map[g]{\Gamma}$ is an end-periodic homotopy equivalence, 
then $\pi_1(Z_g)$ is a free-by-cyclic group, finitely generated by \cref{thm:compactified-mapping-torus}, and isomorphic to $\FBC[g_*]{F_\infty}$ provided $\Gamma$ has at least one end accumulated by loops.
From this perspective, \cref{thm:embedding-thm} shows any fg free-by-cyclic group whose monodromy is induced by the $\pi_1$ action of an end-periodic map can indeed be embedded
in the (fg free)-by-cyclic fundamental group of some finite mapping torus. 
This confirms the conjecture for a large class of groups, though it does not clarify which fg free-by-cyclic groups actually arise in such a manner.

In his paper, Linton removes this uncertainty by proving that a free-by-cyclic group $G = \FBC[\psi]{\FF}$ is finitely generated if and only if $\FF$ has a free product decomposition 
$\FF = A \ast \left(\bigast_{i \in \ZZ} C_i \right)$, where $C_i = \psi^{i}(C_0)$, and where $A$ and $C_0$ are finitely generated \cite[Theorem 1.1]{Linton:EmbeddingFBCGroups}.
This remarkable fact implies the monodromy of every fg group $G = \FBC{F_\infty}$ has an end-periodic topological representative of the following simple description.
Letting $r$ and $s$ denote the rank of $A$ and $C_0$ respectively, its domain is a 2-ended graph,
homeomorphic to a copy of $\RR$ with $s$ circles $c_n^1, \dots, c_n^s$ wedged on at every integer point $n$, 
and an additional $r$ loops $a_1, \dots, a_r$ attached at $n = 0$.
Fixing a basis for $G$ and identifying $\seq{a_i}$ and $\seq{c_0^i}$ with generators of $A$ and $C_0$ respectively, 
the map acts by translating $\RR$ forward one unit, sending $c^i_n$ to $c^i_{n+1}$, and taking $a_i$ to a circuit based at 1 representing $\psi(a_i)$.

This raises several interesting questions about the relationship between fg free-by-cyclic groups and end-periodic mapping classes of infinite graphs.
First, it is clear that two end-periodic mapping tori with the same fundamental group, while homotopy equivalent, need not be properly homotopy equivalent; 
the existence of end-periodic homotopy equivalences of graphs with more than two ends confirms this.
Given a graph of a certain proper homotopy type, one might then ask: which fg $\FBC{F_\infty}$s can be realized by mapping tori of elements from its end-periodic mapping classes?
Alternatively, if $\psi$ is an outer automorphism of $F_\infty$ whose mapping torus is finitely generated, 
can additional algebraic and dynamical properties be leveraged to say more about its end-periodic representatives?
Untangling these connections may show promise for refining of Linton's result (e.g.\ \cite[Question 1.4]{Linton:EmbeddingFBCGroups}).

The geometry of end-periodic mapping tori is also inherently interesting.
By work of Mutanguha \cite[Corollary 5.3.6]{Mutanguha21}, one deduces that the mapping torus of an end-periodic map $g$ is word-hyperbolic if and only if $g$ is atoroidal,
extending Brinkmann's famous result for graph maps of finite rank \cite{Brinkmann}.
Another recent theorem of Linton \cite{Linton:GeometryOfMappingTori} shows that if $g_*$ can be made fully irreducible, $\pi_1(Z_g)$ will be not only hyperbolic, but locally quasi-convex. 
This is a strong property, and devising examples of such maps suggests a way to study these geometrically interesting free-by-cyclic groups from a topological standpoint.

\section{End-periodic graph maps}
\subsection{Cell complexes, graphs, and graph maps}
Let $X$ be a CW complex with $n$-skeleton $\skele[n]{X}$.
We introduce notation $\cells[n]{X}$ for the set of (open) $n$-cells in $X$, elements of which are discrete points if $n = 0$,
or disjoint subsets homeomorphic to open $n$-balls otherwise (here we note the distinction between $\verts{X}$ and $\zskele{X}$ is entirely semantic).
Given an arbitrary subset $U$ of $X$, the collection of $n$-cells from $X$ that lie in $U$ shall likewise be designated $\cells[n]{U} \coloneq \{e \in \cells[n]{X} : e \subeq U\}$.
We write $\ncells[n]{U} \coloneq \abs{\cells[k]{U}}$ for the number of such $n$-cells, noting that a finite-dimensional complex $X = \skele[m]{X}$ is called \term{finite} if $\ncells[n]{X} < \infty$ for all $n$, and \term{infinite} otherwise.
We refer to any closed subset of $X$ which can be realized as a union of cells as a \term{subcomplex}, 
and denote by $\sg{U}$ the largest subcomplex of $X$ contained in $U$ (which is empty if $U$ contains no 0-cells).
Other common notation used throughout the paper includes $\interior U$ for the interior of $U$, $\cl{U}$ for its closure in $X$, and $\fr U$ for the frontier $\cl{U} - \interior U$. 
Note that the term \emph{boundary} will be reserved for distinguished subcomplexes of particular constructions,
to be defined over the course of the paper.

For our purposes, a \term{graph} is a locally finite 1-complex $\Gamma$ 
whose 0-cells $\verts{\Gamma}$ are called \term{vertices} and 1-cells $\edges{\Gamma}$ are called \term{edges}.
Subcomplexes of $\Gamma$ are called \term{subgraphs}.
We use the term \term{path} for any continuous map $\map[\rho]{[0, 1]}{\Gamma}$, and \term{loop} for a path whose image at $0$ and $1$ coincides.
The path $t \mapsto \rho(1 - t)$ which traverses $\rho$ in reverse is designated $\rho^{-1}$.
An \term{oriented edge} is a 1-cell $e \in \edges{\Gamma}$ together with a choice of homeomorphism $\map[\Phi_e]{(0, 1)}{e}$ 
whose extension to $[0, 1] \to \cl{e}$ is a path with endpoints in $\verts{\Gamma}$.

The \term{initial vertex} of $e$ is the point $\Phi_e(0)$, and its \term{terminal vertex} is $\Phi_e(1)$; 
we denote these by $\bd_0 e$ and $\bd_1 e$ respectively.
In practice, we will almost always conflate $e$ and $\Phi_e$, writing $e^{-1}$ for the edge $e$ equipped with the orientation of $\Phi_e^{-1}$, 
and $e_1 \cdot e_2$ for the concatenation $\Phi_{e_1} \cdot \Phi_{e_2}$.
Given a sequence $e_1, \cdots, e_r$ of oriented edges such that $\bd_1 e_i = \bd_0 e_{i + 1}$ for all $i \in \{1, \cdots, r - 1\}$, 
we call $\rho = e_1 \cdot e_2 \cdot \cdots \cdot e_r$ an \term{edge path} of length $r$ with initial vertex 
$\bd_0 \rho \coloneq \bd_0 e_1$ and terminal vertex $\bd_1 \rho \coloneq \bd_1 e_r$.
An edge path is said to be \term{reduced} if it contains no subpaths of the form $e \cdot e^{-1}$, 
and an edge path of length 0 is simply a constant function with image a vertex.

A \term{graph map} is a continuous function $g$ between graphs that sends vertices to vertices and edges to edge paths.
Such a map is called \term{combinatorial} if it takes edges to edges (that is, edge paths of length one), 
and an \term{isomorphism} if it is both combinatorial and a homeomorphism.
We say that $g$ \term{collapses} an edge $e$ if $g(e)$ has length 0.
In the common event that $e$ and $\Phi_e$ are not distinguished, 
we similarly interpret $g(e)$ as the edge path $g \, \Phi_e$.

\subsection{Ends and end-periodicity}
We now introduce the primary objects of study in this paper: end-periodic self maps of infinite graphs.
Let $\Gamma$ be an infinite graph (locally finite by definition).
A sequence of nonempty nested subsets $U_1 \supeq U_2 \supeq \cdots$ of $\Gamma$ is said to be a \term{neighborhood basis} for an end 
if each $U_n$ is a connected component of $\Gamma - K_n$ for some exhaustion of $\Gamma$ by compact subsets $K_1 \subeq K_2 \subeq \cdots$.
Note that this implies $\bigcap_{n} U_n = \emptyset$.
An \term{end} of $\Gamma$ is then an equivalence class of neighborhood bases, where two sequences $\{U_n\}$ and $\{V_n\}$ represent the same end if, for each $n$, there exist $i$ and $j$ such that $U_i \subeq V_n$ and $V_j \subeq U_n$.
The set of ends of $\Gamma$ is denoted $\Ends(\Gamma)$.
An open subset $U$ of $\Gamma$ is said to be a \term{neighborhood} of an end $E$ if it contains the tail of every neighborhood basis for $E$, 
but meets a neighborhood basis of any other end in only finitely many terms.

\begin{definition}[End-periodic maps]
    \label{def:end-periodic}
    Let $\Gamma$ be an infinite graph.
    A graph map $g : \Gamma \to \Gamma$ is called \term{end-periodic} if there exists a neighborhood $U_E$ of each end $E$ such that $g$ restricted to the largest subgraph contained in $\bigcup_{E \in \Ends(\Gamma)} U_E$ is a homeomorphism sending edges to edges, and for some $p > 0$, every $U_E$ satisfies one of the following:
    \begin{itemize}
        \item $g^p(U_E) \subneq U_E$, and $\seq[n \ge 0]{g^{p n}(U_E)}$ is a neighborhood basis for $E$, or
        \item $g^{-p} (U_E) \subneq U_E$, and $\seq[n \ge 0]{g^{-p n}(U_E)}$ is a neighborhood basis for $E$.
    \end{itemize}
\end{definition}

An end neighborhood $U_E$ with the above properties is called a \term{nesting neighborhood} for $E$, 
with the corresponding end labeled as \term{attracting} or \term{repelling} depending on whether the first or second condition is satisfied.
We will use $\Ends_+(\Gamma)$ and $\Ends_-(\Gamma)$ to denote the sets of attracting and repelling ends of $\Gamma$ with respect to a fixed end-periodic map.
The least integer $p$ for which the above conditions are satisfied is called the \term{period} of $g$.

\subsubsection{The action of an end-periodic map on ends}
Let $\map[g]{\Gamma}$ be an end-periodic graph map of period $p$.
Since it behaves like a homeomorphism in a neighborhood of each end, $g$ can be seen to act on $\Ends(\Gamma)$ by sending each end $E$ to the unique end $g(E)$ 
such that $g(U_E) \cap U_{g(E)}$ contains an infinite subgraph.
Since $g^p(E) = E$ for all $E$ by end-periodicity, it follows that $g$ induces an order-$p$ permutation on the ends of its domain.
We write $\Ends(\Gamma) / g$ for the set of orbits under the action of $\gen{\restr{g}{\Ends(\Gamma)}}$.
As we shall see (\cref{prop:wellchosen-system-exists}), this action preserves attracting and repelling ends, 
so $\Ends_+(\Gamma) / g$ and $\Ends_-(\Gamma) / g$ are also well-defined.

The \term{period} of an end $E$ under $g$, denoted $\pd{E}$, is defined to be the least positive integer such that $g^{\pd{E}}(E) = E$.
Since this is the cardinality of its orbit in $\Ends(\Gamma) / g$, the period of each end divides the period $p$ of $g$.
We extend this notion to vertices and edges of $\Gamma$ by declaring the \term{period} 
of any point or subset contained in a nesting neighborhood for $g$ to be the period of its associated end.

\subsubsection{Choosing nesting neighborhoods}
While the definition of end-periodicity guarantees the existence of nesting neighborhoods around the ends of $\Gamma$, 
it says little else about their properties.
It is easy to see that every attracting or repelling end has infinitely many valid nesting neighborhoods; 
indeed, every translate of $U_E$ in its corresponding neighborhood basis also represents a nesting neighborhood for $E$.
In light of this flexibility, it will be useful to introduce some additional properties that a 
practical choice of nesting neighborhoods for an end-periodic map can be expected to satisfy.

\begin{definition}[Well-chosen nesting neighborhoods] \label{def:wellchosen}
    Let $\Gamma$ be a connected graph with finitely many ends, and $\map[g]{\Gamma}$ an end-periodic graph map.
    A collection $\{U_E : E \in \Ends(\Gamma)\}$ of nesting neighborhoods for $g$ is said to be \term{well-chosen} if it satisfies all of the following properties:
    \begin{enumdef}
        \item\label{def:wellchosen.disjoint} Distinct nesting neighborhoods have disjoint closures. 
        \item\label{def:wellchosen.subgraph} The closure of each nesting neighborhood is a subgraph of $\Gamma$. 
        \item\label{def:wellchosen.homeo} Both $g$ and $g^{-1}$ restrict to isomorphisms on the closures of nesting neighborhoods. 
        \item\label{def:wellchosen.leading-ends} Every orbit in $\Ends(\Gamma) / g$ contains a unique end $E$ such that either 
	        \begin{itemize}
	            \item $E$ is attracting, and $g^k(\cl{U_E}) = \cl{U_{g^k(E)}}$ for all $0 \le k < \pd{E}$, or
	            \item $E$ is repelling, and $g^{-k}(\cl{U_E}) = \cl{U_{g^{-k}(E)}}$ for all $0 \le k < \pd{E}$.
	        \end{itemize}
    \end{enumdef}
\end{definition}

We call the end $E$ from \cref{def:wellchosen.leading-ends} the \term{leading end} of its orbit, 
and the associated neighborhood $U_E$ a \term{leading neighborhood}.
Every well-chosen system of nesting neighborhoods determines a collection of leading ends, 
which together form a complete set of representatives for the orbits of $\Ends(\Gamma) / g$.
Although such a choice of representatives is not a canonical feature of $g$, 
it will still be convenient use $\Ends'(\Gamma)$, $\Ends'_+(\Gamma)$, and $\Ends'_-(\Gamma)$ as shorthand for various subsets of leading ends once nesting neighborhoods have been fixed.
The next proposition confirms that such nesting neighborhoods exist in general.

\begin{proposition}
    \label{prop:wellchosen-system-exists}
    Let $\Gamma$ be an infinite graph with finitely many ends.
    Every end-periodic map $\map[g]{\Gamma}$ admits a well-chosen system of nesting neighborhoods.
\end{proposition}
\begin{proof}
    For the following proof, $E$ will represent an attracting end of $\Gamma$.
    In the context of a repelling end, $g$ has a well-defined inverse which may be substituted to obtain an analogous argument.

    Since $g$ is end-periodic, there exists a positive integer $p$ and collection of nesting neighborhoods $\{U_E : E \in \Ends(\Gamma)\}$ satisfying the conditions of \cref{def:end-periodic}.
    As a preliminary observation, we note that if $U_E$ is a nesting neighborhood for $E$, then so is each element of the neighborhood basis $\seq[n \ge 0]{g^{p n}(U_E)}$.
    Since the terms of neighborhood bases of distinct ends are eventually disjoint, and $\Ends(\Gamma)$ is finite, 
    we may replace each $U_E$ with its image under a sufficiently high power of $g^p$ to ensure that all nesting neighborhoods in this collection are pairwise disjoint.

    Parts \refitem{def:wellchosen.subgraph} and \refitem{def:wellchosen.homeo} both follow from the observation that $g^{p N}(U_E) \subeq \sg{U_E}$ for some sufficiently large integer $N$, which exists because the difference $U_E - \sg{U_E}$ is bounded, and therefore contained in one of the compact sets $K_N$ whose complementary components define the neighborhood basis.
    Replacing $U_E$ with $g^{p (N + 1)}(U_E)$ guarantees $g$ as well as $g^{-1}$ restrict to homeomorphisms on a subgraph of $\Gamma$ containing $U_E$.
    As such, we may expand $U_E$ by the union of open edges whose intersection with $U_E - \sg{U}_E$ is nonempty to ensure $\cl{U_E}$ is a subgraph.

    With property \refitem{def:wellchosen.homeo} established, it becomes clear that $g$ preserves attracting and repelling ends of $\Gamma$.
    Indeed, if $E$ is an attracting end with associated neighborhood basis $\seq[n \ge 0]{g^{n p}(U_E)}$, 
    the homeomorphic images of these sets under $g$ constitute a neighborhood basis $\seq[n \ge 0]{g^{n p + 1}(U_E)}$ for $g(E)$, 
    which must therefore be attracting.
    By once again replacing $U_E$ with its image under a power of $g^p$, 
    we may assume $g^k(\cl{U_E}) \subeq \cl{U_{g^k (E)}}$ holds for all $1 \le k < \pd{E}$, 
    after which $\{U_E, g(U_E), \cdots, g^{\pd{E} - 1}(U_E)\}$ represent nesting neighborhoods for ends in the $g$-orbit of $E$ that satisfy part \refitem{def:wellchosen.leading-ends}.
    Repeating this process for each orbit in $\Ends_+(\Gamma) / g$, 
    and doing likewise with $g^{-1}$ for orbits of repelling ends, the resulting nesting neighborhoods form a well-chosen system.
\end{proof}

\begin{example}
    \begin{figure}[htb]
        \centering
        \includegraphics[width=\textwidth]{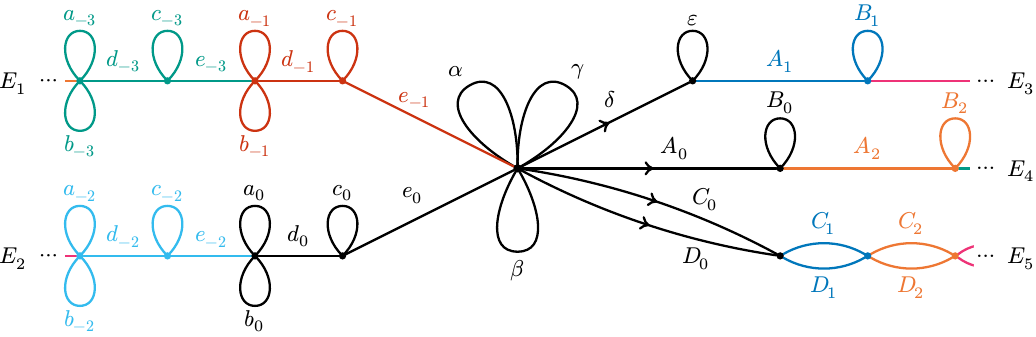}
        \captionsetup{singlelinecheck=off}
        \caption[.]{
            Let $\Gamma$ be the infinite, 5-ended graph depicted above, and $\Gamma_0$ the finite subgraph whose vertices and edges are shaded black.
            We define $g$ to act on a subset of edges in $\Gamma$ by
            \begin{equation*}
                g: \left\{
                \begin{aligned}
                    a_0 & \mapsto \alpha              & \alpha   & \mapsto \alpha \cdot \delta \cdot \epsilon \cdot \delta^{-1} \\
                    b_0 & \mapsto \beta               & \beta    & \mapsto C_0 \cdot D_0^{-1}                                   \\
                    c_0 & \mapsto \gamma              & \gamma   & \mapsto A_0 \cdot B_0 \cdot A_0^{-1}                         \\
                    d_0 & \mapsto \alpha \cdot \gamma & \delta   & \mapsto A_0                                                  \\
                    e_0 & \mapsto \beta               & \epsilon & \mapsto B_0
                \end{aligned}
                \right.
            \end{equation*}
            while mapping the others $x_{-i} \mapsto x_{-i + 1}$ and $X_{i - 1} \mapsto X_i$ for all positive $i$, where $x$ can be any one of the symbols $\{a, b, c, d, e\}$, and $X$ is likewise in $\{A, B, C, D\}$.
        }
        \label{fig:Gamma}
    \end{figure}
    Consider the map $\map[g]{\Gamma}$ defined in \cref{fig:Gamma}, which will serve as a running example throughout the paper.
    This is an end-periodic homotopy equivalence of period $2$, with repelling ends $\Ends_-(\Gamma) = \{E_1, E_2\}$ 
    and attracting ends $\Ends_+(\Gamma) =\{E_3, E_4, E_5\}$.
    The components of $\Gamma - \Gamma_0$ form a system of well-chosen nesting neighborhoods.
    Because $g$ swaps $E_1$ and $E_2$, it follows that $\pd{E_1} = \pd{E_2} = 2$, and $\Ends_-(\Gamma) / g$ consists of a single orbit with leading end $E_1$.
    On the other hand, our attracting ends with period $\pd{E_3} = \pd{E_4} = 2$ and $\pd{E_5} = 1$ are partitioned into two orbits $\Ends_+(\Gamma) / g = \{\{E_3, E_4\}, \{E_5\}\}$, whose leading ends are $E_3$ and $E_5$.
\end{example}

\subsection{Decompositions induced by end-periodic maps}
Once nesting neighborhoods are fixed, an end-periodic map can be seen to induce several decompositions of its underlying graph.
The most immediate of these is the decomposition into a core and nesting domains.

\begin{convention} \label{conv:g}
    For the remainder of the paper, unless otherwise noted, $\Gamma$ will represent an infinite connected graph with finitely many ends, 
    and $\map[g]{\Gamma}$ an end-periodic graph map.
    Consistent with previous usages, we write $U_E$ to refer to a generic nesting neighborhood for some end $E$ under $g$.
    It is implicit that any such $U_E$ belongs to a system of nesting neighborhoods $\{U_E : E \in \Ends(\Gamma)\}$, which, 
    in light of \cref{prop:wellchosen-system-exists}, we assume to be well-chosen.
\end{convention}

\begin{definition}[Nesting domains and core]
    Fix a set $\cU = \set{U_E : E \in \Ends(\Gamma)}$ of well-chosen nesting neighborhoods for $g$.
    The \term{nesting domains} $\Gamma_+$ and $\Gamma_-$ associated to $\cU$ are the unions
    \begin{align*}
        \Gamma_+ & \coloneq \bigcup_{E \in \Ends_+(\Gamma)} U_E, & \Gamma_- & \coloneq \bigcup_{E \in \Ends_-(\Gamma)} U_E.
    \end{align*}
    The \term{core} for $g$ assocated to $\cU$ is the complement of these nesting domains:
    \[
        \Gamma_0 \coloneq \Gamma - (\Gamma_+ \cup \Gamma_-).
    \]
\end{definition}
Since nesting neighborhoods are well-chosen, \cref{def:wellchosen.subgraph} implies $\cl{\Gamma_+}$ and $\cl{\Gamma_-}$ are subgraphs,
the components of which must be closures of nesting neighborhoods by \cref{def:wellchosen.disjoint}.
The core $\Gamma_0$ is also a subgraph: finite and nonempty because $\Gamma_+$ and $\Gamma_-$ have disjoint closures while containing a neighborhood of every end, 
and connected because each $U_E$ is non-separating, being part of a neighborhood basis.

\begin{remark}
    In general, we will refer to any finite subgraph of $\Gamma$ as a \term{core} for $g$ if its complement is a (disjoint) union of well-chosen nesting neighborhoods.
    While specifying a core is often more expedient than listing nesting neighborhoods, the notions are dual and can be used interchangeably.
\end{remark}

\begin{definition}[Junctures]
    The \term{positive and negative junctures} of a core $\Gamma_0$ are the sets 
    \begin{align*}
        \bd_+ \Gamma_0 & \coloneq \fr \Gamma_+ = \Gamma_0 \cap \cl{\Gamma_+}, & \bd_- \Gamma_0 & \coloneq \fr \Gamma_- = \Gamma_0 \cap \cl{\Gamma_-}.
    \end{align*}
\end{definition}
Since $\cl{\Gamma_\pm}$ is a subgraph while $\Gamma_0 \cap \Gamma_\pm = \emptyset$, 
each juncture is a (finite) subset of $\verts{\Gamma_0}$ whose elements we call \term{juncture vertices}.

\subsubsection{Block decompositions}
A core for $g$ is naturally associated with a partition of the underlying graph 
that comes from translating nesting domains under successive powers of $g$ and $g^{-1}$.
The next definition makes this explicit.

\begin{definition}[Block decomposition]
    Let $\Gamma_0$ be a for $g$.
    The \term{block decomposition} associated to $\Gamma_0$ is the bi-infinite sequence of subsets $\seq[n \in \ZZ]{B_n}$ obtained by setting
    \begin{align*}
        B_0 & \coloneq \Gamma_0, & B_1 & \coloneq \Gamma_+ - g(\Gamma_+), & B_{-1} & \coloneq \Gamma_- - g^{-1}(\Gamma_-),
    \end{align*}
    and then inductively defining $B_\pmn \subeq \Gamma_\pm$ for all $n > 1$ to be the $(n - 1)$th translate of $B_{\pm 1}$ under $g^{\pm 1}$:
    \begin{align*}
        B_n & \coloneq g^{n - 1}(B_1), & B_{-n} & \coloneq g^{-n + 1}(\Gamma_-).
    \end{align*}
    We refer to each $B_n$ as a \term{block}, with $B_0$ being the \term{zero block}, and others called \term{positive or negative blocks} depending on the sign of $n$.
\end{definition}

\begin{remark}
    Since the restriction $\restr{g}{\Gamma_\pm}$ is a homeomorphism by end-periodicity, each nonzero block is homeomorphic to either $B_1$ or $B_{-1}$ depending on its sign.
    Unwrapping the inductive definition above, one obtains the following equivalent characterization of nonzero blocks $B_\pmn$, $n > 0$:
    \[
        B_n \coloneq g^{n - 1}(\Gamma_+) - g^n(\Gamma_+), \hspace{6em} B_{-n} \coloneq g^{-n + 1}(\Gamma_-) - g^{-n}(\Gamma_-).
    \]
    As $\Gamma$ is connected, we note that each nonzero block is neither open nor closed.
\end{remark}

Block decompositions are a key tool for understanding the combinatorial structure of $\Gamma$, 
and the following properties, which come easily from the definition, will be used abundantly in subsequent sections.

\begin{proposition}
    Let $\seq[n \in \ZZ]{B_n}$ be the block decomposition associated to a well-chosen core for $g$.
    Then the following hold.
    \begin{enumprop}
        \item \label{prop:blockprops.closure-subgraph} $\cl{B_n}$ is a finite subgraph of $\Gamma$ for all $n \in \ZZ$.
        \item \label{prop:blockprops.decomposition} $\Gamma_+ = \bigcup_{n \ge 1} B_n$ and $\Gamma_- = \bigcup_{n \ge 1} B_{-n}$, each a disjoint union. 
            Thus, $\Gamma = \bigcup_{n \in \ZZ} B_n$. \null\hfill\qedsymbol
    \end{enumprop}
\end{proposition}

\begin{example}
    Consider the map $g$ defined in \cref{fig:Gamma}.
    The subgraph $\Gamma_0$ is a core for $g$, and components of $\Gamma - \Gamma_0$ are thus a well-chosen family of nesting neighborhoods.
    Vertices and edges outside $\Gamma_0$ are shaded to indicate its associated block decomposition, 
    with cells of the same color belonging to the same block.
\end{example}

\subsubsection{The cell structure of $\Gamma$} \label{sec:cell-structure-Gamma}
Since disjoint blocks partition both the vertices and open edges of the graph $\Gamma$, they give us a means to  fully enumerate its cell structure.
The following notation for cells of $\Gamma$ relative to a particular block decomposition $\seq[n \in \ZZ]{B_n}$ will prove quite useful later on.

For $\ep \in \{-1, 0, 1\}$, begin by listing vertices of $B_\ep$ as
\[
    \verts{B_\ep} = \{v_\ep^j : 1 \le j \le \nverts{B_\ep}\}.
\]
Now, for all $n > 1$, set $v_n^j \coloneq g^{n - 1}(v_1^j)$ and $v_{-n}^j \coloneq g^{-n + 1}(v_{-1}^j)$.
Since $B_n$ and $B_{-n}$ are homeomorphic to $B_1$ and $B_{-1}$ via $g^{n - 1}$ and $g^{-n + 1}$ respectively, their vertices are exactly
\begin{align*}
    \verts{B_n} & = \{v_n^j : 1 \le j \le \nverts{B_1}\}, & \verts{B_{-n}} & = \{v_{-n}^j : 1 \le j \le \nverts{B_{-1}}\}.
\end{align*}
The entire 0-skeleton of $\Gamma$ can therefore be enumerated
\[
    \verts{\Gamma} = \{v_n^j : n \in \ZZ, 1 \le j \le \nverts{B_{\sgn n}}\}.
\]

We use the same approach to catalog edges.
Starting with
\[
    \edges{B_\ep} = \{e_\ep^i : 1 \le j \le \nedges{B_\ep}\},
\]
put $e_n^i \coloneq g^{n - 1}(e_1^i)$ and $e_{-n}^i \coloneq g^{-n + 1}(e_{-1}^i)$ for all $n > 1$.
This gives
\begin{align*}
    \edges{B_n} & = \{e_n^i : 1 \le i \le \nedges{B_1}\}, & \edges{B_{-n}} & = \{e_{-n}^i : 1 \le i \le \nedges{B_{-1}}\},
\end{align*}
so that the complete set of 1-cells of $\Gamma$ can be written
\[
    \edges{\Gamma} = \{e_n^i : n \in \ZZ, 1 \le i \le \nedges{B_{\sgn n}}\}.
\]

When it is necessary to orient the edges of $\Gamma$, we will always do so in a way that ensures $g$ is orientation preserving outside its core.
This allows us to operate with the assumption that $g^{\pm 1}(\bd_\ell e_n^i) = \bd_\ell e_{n \pm 1}^i$ for all $n \neq 0$ and $\ell \in \{0, 1\}$.

\subsubsection{Components of blocks}

Sometimes, whole blocks may be too coarse for the task at hand, and we will want to distinguish their individual components.
These can be indexed by leading ends, as the next proposition indicates.
\begin{proposition} \label{prop:Bn-components}
    Let $\Gamma_0$ be a core for $g$ with associated block decomposition $\seq[n \in \ZZ]{B_n}$.
    The components of each nonzero block $B_\pmn$, $n \ge 1$ are in bijection orbits of $\Ends_\pm(\Gamma) / g$.
    In particular, the components of $B_{\pm 1}$ are
    \[
        \pi_0(B_{\pm 1}) = \set{B_{\pm 1} \cap U_E = U_E - g^{\pm \pd{E}}(U_E) : E \in \Ends'_\pm(\Gamma)},
    \]
    with components of $B_\pmn$ being translates
    \[
        \pi_0(B_\pmn) = \set{g^{\pm(n - 1)}(B_{\pm 1} \cap U_E) = B_\pmn \cap U_{g^{\pm(n - 1)}(E)} : E \in \Ends'_\pm(\Gamma)}.
    \]
\end{proposition}

\begin{proof}
    \Cref{def:wellchosen.leading-ends} allows each nesting neighborhood of an attracting end to be written as
    $g^k(U_E)$ for some leading end $E \in \Ends'_+(\Gamma)$ and power $0 \le k < \pd{E}$.
    Consequently, $\Gamma_+$ can be expressed as the union
    \[
        \Gamma_+ = \bigcup_{E \in \Ends'_+(\Gamma)} \bigcup_{k = 0}^{\pd{E} - 1} g^k(U_E),
    \]
    whose terms are disjoint by \cref{def:wellchosen.disjoint}.
    Applying this expansion to $B_1$, we find
    \begin{align*}
        B_1 = \Gamma_+ - g(\Gamma_+) & = \bigcup_{E \in \Ends'_+(\Gamma)} \left( \bigcup_{k = 0}^{\pd{E} - 1} g^k(U_E) - \bigcup_{k = 1}^{\pd{E}} g^k(U_E) \right) \\
                                     & = \bigcup_{E \in \Ends'_+(\Gamma)} U_E - g^{\pd{E}}(U_E)                                                                    \\
                                     & = \bigcup_{E \in \Ends'_+(\Gamma)} B_1 \cap U_E,
    \end{align*}
    with the final equality justified by the observation that $g^{\pd{E}}(U_E)$ is the only component of $g(\Gamma_+)$ to meet $U_E$.
    The analogous argument for $\Gamma_-$ gives
    \[
        B_{-1} = \bigcup_{E \in \Ends'_-(\Gamma)} U_E - g^{-\pd{E}}(U_E) = \bigcup_{E \in \Ends'_-(\Gamma)} B_{-1} \cap U_E.
    \]
    Thus, the components of $B_1$ and $B_{-1}$ are in bijection with the leading attracting and repelling ends of $\Gamma$, and can be enumerated as in the proposition.
    Since a nonzero block $B_\pmn$ is homeomorphic to $B_{\pm 1}$ via $g^{\pm(n - 1)}$, the rest follows.
\end{proof}

\subsubsection{Proper cores}
We will say that a core $\Gamma_0$ for $g$ is \term{proper} if it contains the image of every edge mapped non-combinatorially.
Morally, such a core is as close to being invariant under $g$ as one can ask, satisfying the identity $g(B_{-1} \cup \Gamma_0) = \Gamma_0 \cup B_1$,
where $B_{\pm 1}$ denote blocks of its corresponding decomposition.
We can always upgrade a well-chosen core to a proper core by enlarging it in the following way.

\begin{definition}[Core enlargement]
    Let $\Gamma_0$ be a core for $g$ with associated block decomposition $\seq[n \in \ZZ]{B_n}$.
    The \term{$n$th enlargement} of $\Gamma_0$ is the set $\Gamma_n \coloneq \bigcup_{i = - n}^n B_i$.
\end{definition}

It is easily verified that any enlargement of $\Gamma_0$ is itself core for $g$.
Thus,

\begin{proposition}
    \label{prop:big-core}
    If \(\Gamma_0\) is a core for \(g\), then some enlargement of \(\Gamma_0\) is a proper core.
\end{proposition}

\begin{proof}
    Since the blocks $\seq[n \in \ZZ]{B_n}$ cover $\Gamma$, the sequence $\Gamma_0 \subeq \Gamma_1 \subeq \cdots$ represents an exhaustion of our graph by compact sets.
    It follows by the continuity of $g$ that $g(\Gamma_0) \subeq \Gamma_n$ for all $n$ sufficiently large.
    Every such $\Gamma_n$ is a core for $g$, and must be proper since the map acts combinatorially on edges of $\Gamma_n - \Gamma_0$.
\end{proof}

\subsubsection{Subgraph and joining edges}
\label{sec:standard-orientation}

The block decomposition $\seq[n \in \ZZ]{B_n}$ associated to a choice of core for $g$ gives rise to the following trichotomy between edges of $\Gamma$.

\begin{definition}[Subgraph and joining edges]
    A 1-cell in $\edges{\Gamma}$ will be called a \term{core edge} if it lies in $\Gamma_0$,
    a \term{subgraph edge} if it is contained in the largest subgraph $\sg{B_n}$ of some nonzero block $B_n$, 
    and a \term{joining edge} otherwise.
\end{definition}

For each $n \neq 0$, we put $J_n \coloneq B_n - \sg{B_n}$.
Since $\cl{B_n}$ is itself a subgraph, we see that $J_n$ is exactly the union of joining edges in $B_n$, so named because they join vertices of $\sg{B_n}$ to vertices of an adjacent block $\sg{B_k}$.
Since $g$ is a homeomorphism outside the core, we have $J_{\pm(n+1)} = g^{\pm1}(J_{\pmn})$ for all $n \ge 1$.
Thus, every edge of $J_\pmn$ is the translate under $g^{\pm(n - 1)}$ of an edge from $B_{\pm1}$ with a vertex in $\bd_\pm \Gamma_0$.

We will say an edge of $J_n$, $n \ne 0$ is in \term{standard orientation} if it has its terminal vertex in $\sg{B_n}$ and initial vertex in $\sg{B_k}$, 
where the value of $k$ depends on its period $q$:
\[
    k = \begin{cases}
        \max\{0, n - q\} & \text{if } n > 0  \\
        \min\{0, n + q\} & \text{if } n < 0.
    \end{cases}
\]
Going forward, we shall assume joining edges are in standard orientation as a matter of convention.

\begin{example}
    In our running example (\cref{fig:Gamma}), $J_n = \set{A_n, C_n, D_n}$ and $J_{-n} = \set{e_{-n}}$ for all $n \ge 1$.
\end{example}

\begin{remark}
    \label{rmk:at-least-one-joining-edge-per-component}
    The distinction between subgraph and joining edges is not a canonical feature of $g$, since it depends quite significantly on a choice of block decomposition.
    However, because a nonzero block $B_n$ of any decomposition properly contains $\sg{B_n}$ (the latter being closed while the former is not), it holds in general that $J_n$ is nonempty for all $n \neq 0$.
    In fact, since the individual components of $B_n$ are non-closed by \cref{prop:Bn-components}, it must be the case that every component of a nonzero block contains at least one joining edge.
\end{remark}

\section{End-periodic mapping tori}

\subsection{The mapping torus of a graph map}
Let $\gX$ be a connected graph and $\map[f]{\gX}$ a graph map.
The \term{mapping torus} of $f$, denoted $Z_f$, is the defined as the quotient of $\gX \times [0, 1]$ 
under the equivalence relation $\sim$ generated by declaring $(x, 1) \sim (f(x), 0)$ for all $x \in \gX$.

Recall that $Z_f$ is naturally a 2-complex whose cell structure descends from the following canonical cellulation of $\gX \times [0, 1]$. 
Viewing the subsets $\gX \times \{0\}$ and $\gX \times \{1\}$ as disjoint copies of the original graph, the 0-skeleton of $\gX \times [0, 1]$ taken to be the union of their vertices: $\skele[0]{\gX} \times \{0, 1\}$.
For each $v \in \verts{\gX}$, the set of points $v \times (0, 1)$ constitutes the interior of a 1-cell attached from $(v, 0)$ to $(v, 1)$, and these, along with the two copies of the original graph comprise the 1-skeleton of $\gX \times [0, 1]$.
Finally, for all edges $e$ of $\gX$, the points $e \times (0, 1)$ define the interior of a 2-cell of $\gX \times [0, 1]$ which attaches along the loop in the 1-skeleton based at $\bd_0 e$ and traversing the 1-cells $e \times \{0\}$, $\{\bd_1 e\} \times [0, 1]$, $e^{-1} \times \{1\}$, and $\{\bd_0 e\} \times [0, 1]^{-1}$ in sequence.

In order for the cell structure on $\gX \times [0, 1]$ to pass to the quotient, it is necessary that the equivalence relation respects preserves $n$-skeleta.
Since $Z_f$ is obtained by gluing edges of $\gX \times \{1\}$ along edge paths of $\gX \times \{0\}$, one can ensure this by subdividing edges of $\gX \times \{1\}$, 
adding a new 0-cell at the point $v \times \{1\}$ for all $v \in f^\inv(\skele[0]{\gX}) - \skele[0]{\gX}$.
With this refinement, $\sim$ identifies vertices with vertices and edges with edges, allowing the cellulation of $\gX \times [0, 1]$ to descend to a cellulation of $Z_f$.
There is a canonical isomorphism between the original graph $\gX$ and the subcomplex $\gX \times \{0\} / \mathord{\sim}$ inside the 1-skeleton of $Z_f$; 
we will routinely identify the two in order to view $\gX$ as a subset of $Z_f$.
Because the equivalence relation only identifies points of the 1-skeleton, 
the mapping torus contains for each $e \in \edges{\gX}$ a 2-cell attaching along the loop $e \cdot t_0 \cdot f(e)^{-1} \cdot t_1^{-1}$, 
as shown in \cref{fig:rectangle}.
\begin{figure}[htpb]
    \centering
    \includegraphics[]{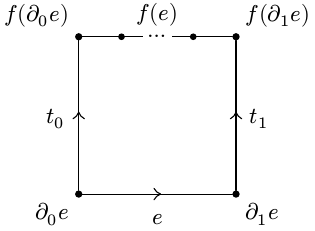}
    \caption{ Generic 2-cell in the mapping torus $Z_f$. }
    \label{fig:rectangle}
\end{figure}

$Z_f$ is equipped with a \term{suspension semiflow} $\map[\psi_f]{Z_f \times [0, \infty)}{Z_f}$, 
defined as the quotient of the local semiflow on $\gX \times [0, 1]$ that maps $((x, y), t) \mapsto (x, y + t)$ for all $x \in \gX$, $0 \le y \le 1$, and $0 \le t \le 1 - y$.
The \term{forward orbit} of a point $z \in Z_f$ under this semiflow is the subset $\psi_f(z, [0, \infty))$, 
and we use $\map[\psi_f^t]{Z_f}$ to denote the mapping $z \mapsto \psi_f (z, t)$.
Note that our parameterization of the semiflow is such that $\psi_f^n = f^n$ for all $n \ge 0$.

If $f$ does not collapse edges, then every 2-cell of its mapping torus attaches along a path like the one in \cref{fig:rectangle}.
We abide by the following convention for depicting such 2-cells in all diagrams that follow.

\begin{convention}
    Let $e$ be an oriented edge not collapsed by $f$, identified with $(0, 1)$ via the homeomorphism $\Phi_e$.
    We will depict the closure of the 2-cell $\psi_f(e, (0, 1))$ as the unit square in $\RR^2$,
    first identifying its interior with $(0, 1)^2$ by and $\psi_f(x, t) \mapsto (\Phi_e(x), t)$ for all $x \in e$, $t \in (0, 1)$, 
    and then continuously extending to $[0, 1]^2$.
    This sends $e$ and $f(e)$ to the horizontal segments $(0, 1) \times \{0\}$ and $(0, 1) \times \{1\}$ respectively, 
    while mapping $\psi_f(\bd_\ell e, [0, 1])$ to the vertical segment $\{\ell\} \times [0, 1]$ for $\ell \in \{0, 1\}$.
    Note that in $Z_f$, pieces of these perimeter edges may become identified.
\end{convention}

In accordance with our convention for depicting them, we shall adopt the term \term{vertical} 
to describe 1-cells of the form $v \times (0, 1) / \mathord{\sim}$ in the mapping torus, 
and \term{horizontal} for those descended from edges of $\gX \times \{0\}$.
It follows from the cellulation of $\gX \times [0, 1]$ that every 1-cell of $Z_f$ is either horizontal or vertical.

\subsection{Mapping tori of end-periodic graph maps}
Let $\map[g]{\Gamma}$ be an end-periodic map as in \cref{conv:g}.
Letting $Z = Z_g$ denote its mapping torus, we preemptively identify $\Gamma$ within the 1-skeleton of $Z$.

\subsubsection{Positive and negative mapping tori}
\label{sec:positive-and-negative-mapping-tori}

The decomposition $\Gamma = \Gamma_0 \cup \Gamma_+ \cup \Gamma_-$ induced by a core $\Gamma_0$ for $g$ induces the following decomposition of its mapping torus.
\begin{definition}[Positive and negative mapping tori]
    The \term{positive and negative mapping tori} of $g$ are the subsets $Z_+$ and $Z_-$ of $Z$ corresponding to the mapping tori of $\restr{g}{\Gamma_+}$ and $\restr{g^{-1}}{\Gamma_-}$ respectively.
    The set $Z_0$ is the difference $Z_0 \coloneq Z - (Z_+ \cup Z_-)$.
\end{definition}

The fact that $\ginv$ is not only well-defined on $\Gamma_-$, but takes this domain into itself properly and homeomorphically means
that each flowline of $\psi_g$ passing through a point of $\Gamma_-$ can be extended infinitely backwards.
This yields a semiflow on the negative mapping torus, denoted $\map[\psi_\ginv]{Z_- \times \halfline}{Z_-}$,
which is the inverse of $\psi_g$ in the sense that $\psi_g^t \psi_\ginv^t = \id_{Z_-}$ for all $t \ge 0$.
We can thus interpret $Z_+$ and $Z_-$ dynamically as the forward orbits of $\Gamma_+$ and $\Gamma_-$ under $\psi_{g}$ and $\psi_{g^{-1}}$ respectively.

While $Z_0$, $Z_+$, and $Z_-$ are themselves neither open nor closed, their closures are all subcomplexes of $Z$.
In particular, $\cl{Z_0}$ is finite, while each $\cl{Z_\pm}$ is infinite.
By analogy to the junctures of $\Gamma$, we let $\bd_\pm Z_0 \coloneq \cl{Z_0} \cap \cl{Z_\pm}$ denote the finite 1-subcomplexes of $Z$ 
where the closures of $Z_\pm$ meet $\cl{Z_0}$.
The definition of the mapping torus shows each $\bd_\pm Z_0$ consists of the vertical 1-cells $\psi_g(v, [0, 1])$ corresponding to juncture vertices $v \in \bd_\pm \Gamma_0$,
and the horizontal 1-cells from $\Gamma_0 - g(\Gamma_0)$ or $g(\Gamma_0) - \Gamma_0$.
If $\Gamma_0$ is a proper core, we see that the horizontal 1-cells of $\bd_\pm Z_0$ are exactly $\edges{B_{\pm 1}}$.

Whereas $\Gamma_\pm$ has one component for each attracting or repelling end of $\Gamma$, the corresponding mapping torus $Z_\pm$
has a component for each orbit of $\Ends_\pm(\Gamma)$ under the action of $g$.
Indeed, by \cref{prop:Bn-components}, we can view each component of $Z_\pm$ as the forward orbit of a component
$B_{\pm 1} \cap U_E$ of $B_{\pm 1}$ under $\psi_{g^{\pm 1}}$.

\subsubsection{The cell structure of $Z$}
As was done for vertices and edges of $\Gamma$ in \cref{sec:cell-structure-Gamma}, 
one can enumerate cells of $Z$ along the lines of the decomposition $Z_0 \cup Z_+ \cup Z_-$.
Since $Z_0$ contains only finitely many cells, the interesting task here is to describe the structures of the positive and negative mapping tori.

We begin with $Z_+$, whose 0-cells, canonically identified with the vertices of $\Gamma_+$, are enumerated as in \cref{sec:cell-structure-Gamma}, and whose horizontal 1-cells likewise correspond to edges of $\Gamma_+$.
Vertical 1-cells of $Z_+$ exist in bijection with $\verts{Z_+}$, each being a segment of a flowline of $\psi_g$ that originates from a vertex, and for all $n \ge 1$, we let $t_n^j$ denote the vertical edge $\psi_g(v_n^j, (0, 1))$ whose attaching map is $\bd_\ell t_n^j = v_{n + \ell}^j, \ell \in \{0, 1\}$.
Thus,
\[
    \edges{Z_+} = \{e_n^i : 1 \le i \le \nedges{B_1}, n \ge 1\} \cup \{t_n^j : 1 \le j \le \nverts{B_1}, n \ge 1\}.
\]

Elements of $\faces{Z_+}$ share a similar correspondence with edges of $\Gamma_+$, where for each horizontal edge $e_n^i \in \edges{Z_+}$, 
we have a 2-cell $f_n^i$ given by the partial orbit $\psi_g(e_n^i, (0, 1))$.
This allows us to write
\[
    \faces{Z_+} = \{f_n^i : 1 \le i \le \nedges{B_1}, n \ge 1\}.
\]
Because the action of $g$ on $\Gamma_+$ is combinatorial, every such $f_n^i$ attaches along a square of 1-cells as illustrated in \cref{fig:2-cells-Zplus}.
Here we note that, while the indexing scheme for vertices requires us to specify this attachment differently based on whether $e_n^i$ is a subgraph or joining edge, 
the result in both cases is structurally identical.

This gives an exhaustive account of the cell structure of $Z_+$.
Cells of $Z_-$ are named by repeating this process with $\Gamma_-$, $\ginv$, and $\psi_\ginv$ filling the roles of $\Gamma_+$, $g$, and $\psi_g$ respectively.

\begin{figure}[htb]
    \centering
    \begin{subfigure}[b]{0.45\textwidth}
        \centering
        \includegraphics[]{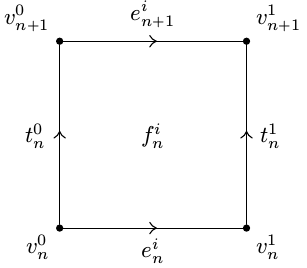}
        \subcaption{$e_n^i$ a subgraph edge.}
        \label{fig:2-cells-Zplus:f}
    \end{subfigure}
    \hfill
    \begin{subfigure}[b]{0.45\textwidth}
        \centering
        \includegraphics[]{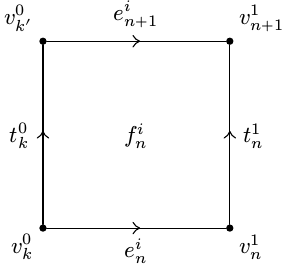}
        \subcaption{$e_n^i$ a joining edge.}
        \label{fig:2-cells-Zplus:h}
    \end{subfigure}
    \caption{Attaching map of the 2-cell $f_n^i$ in the case where $e_n^i$ is a subgraph edge (left), or a joining edge (right).
        In the latter diagram, $k = \max\{0, n - q\}$ and $k' = \max\{0, n - q + 1\}$, where $q$ is the period of $e_n^i$.}
    \label{fig:2-cells-Zplus}
\end{figure}

\subsection{Compactification}
When $f$ is an end-periodic homeomorphism of an infinite-type surface, it is well-known 
that its mapping torus $Z_f$ admits a flowline-preserving compactification that realizes it as the interior of 
a compact 3-manifold whose boundary is a finite union of closed surfaces \cite{Fenley97}.
The aim of this section is to prove an analogous result for mapping tori of end-periodic graph maps, namely:

\compactifiedMappingTorusTheorem

An immediate corollary is that $\pi_1(Z_g)$ is finitely generated.
The proof is in two parts.
Beginning with the mapping torus $Z = Z_g$, we first describe a procedure, called \emph{compactification along flowlines},
which realizes $W_g$ as the union of $Z$ and a certain boundary at infinity.
Then, we adapt the natural cell structure of $Z$ to obtain a cellulation of $W_g$.
A visual depiction of the process is provided in \cref{fig:compactification}.

\begin{figure}[ht]
    \centering
    \begin{subfigure}[b]{0.1\textwidth}
        \includegraphics[width=\textwidth]{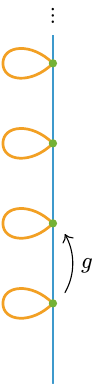}
        \subcaption{$\Gamma_+$}
        \label{fig:compactification-1}
    \end{subfigure}
    \hfill
    \begin{subfigure}[b]{0.29\textwidth}
        \includegraphics[width=\textwidth]{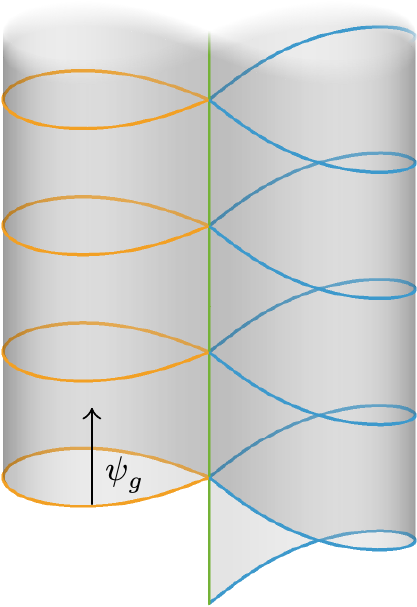}
        \subcaption{$\cl{Z_+}$}
        \label{fig:compactification-2}
    \end{subfigure}
    \hfill
    \begin{subfigure}[b]{0.29\textwidth}
        \includegraphics[width=\textwidth]{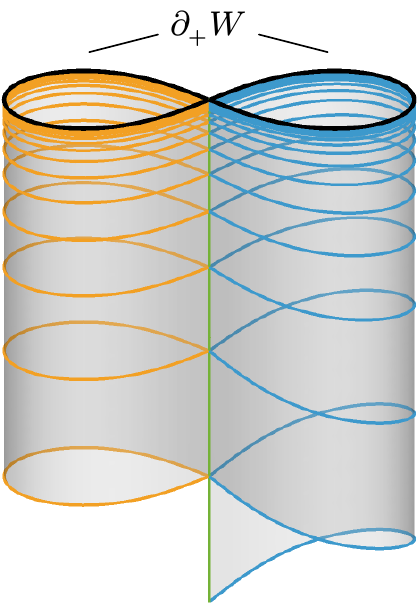}
        \subcaption{$\cl{W_+}$}
        \label{fig:compactification-3}
    \end{subfigure}
    \hfill
    \begin{subfigure}[b]{0.29\textwidth}
        \includegraphics[width=\textwidth]{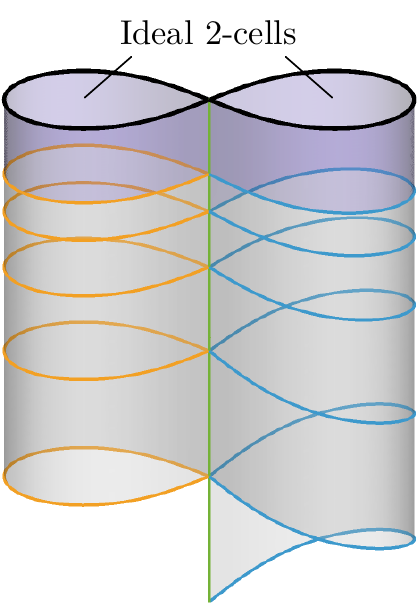}
        \subcaption{Cellulation of $\cl{W_+}$}
        \label{fig:compactification-4}
    \end{subfigure}
    \caption{
        Compactification and cellulation of a positive mapping torus.
        Let $\map[g]{\Gamma}$ be an end-periodic map whose action on $\Gamma_+$, depicted in (\subref{fig:compactification-1}), is translation by one vertex.
        Each positive block of $\Gamma$ has one joining edge (blue), and one subgraph edge (orange).
        \Cref{fig:compactification-2} shows $\Gamma_+$ embedded in the positive mapping torus $Z_+$, whose closure is an infinite 2-complex.
        Flowlines of $\psi_g$ correspond to vertical rays, and after compactifying along them we are left with the space $\cl{W_+}$ presented in (\subref{fig:compactification-3}); 
        note that $\Gamma_+$ is no longer embedded.
        By truncating $\Gamma_+$ at some cutoff, we achieve a finite cell structure on $\cl{W}$ as in (\subref{fig:compactification-4}).
    }
    \label{fig:compactification}
\end{figure}

\subsubsection{Compactification along flowlines}
Fix a core $\Gamma_0$ for $g$, and identify $\Gamma_\pm$ in the 1-skeleton of $Z_\pm$ as usual.
Being contained in a mapping torus, the forward orbit $\gamma$ of a point in $Z_+$ under $\psi_g$ is fully determined by its intersections with $\Gamma_+$, of which there are infinitely many.
Since $\restr{g}{\Gamma_+}$ is a homeomorphism with no fixed points, it follows that $\gamma$ is an infinite ray, which must be unbounded because $\bigcap_{n \ge 0} g^n(\Gamma_+) = \emptyset$.
Forward orbits of points in $Z_-$ under $\psi_\ginv$ are likewise rays that escape compact sets.

Roughly speaking, the compactification of $Z$ is achieved by appending to every such ray an ideal point.
We begin by defining two disjoint sets $\bd_\infty Z_+$ and $\bd_\infty Z_-$, to be made up of ideal points
originating from flowlines of $\psi_g$ through $Z_+$, and flowlines of $\psi_\ginv$ through $Z_-$ respectively.
Formally, the elements of $\bd_\infty Z_\pm$ are equivalence classes of forward orbits $\gamma_z \coloneq \psi_{g^{\pm 1}}(z, \halfline)$, $z \in Z_\pm$,
where two rays $\gamma_w, \gamma_z \in \bd_\infty Z_\pm$ represent the same point if
$\gamma_w \subeq \gamma_z$ or $\gamma_z \subeq \gamma_w$. 
Note that a single flowline of $\psi_g$ intersecting both $Z_+$ and $Z_-$ represents distinct points of $\bd_\infty Z_+$ and $\bd_\infty Z_-$.

Letting $W_\pm$ denote the union $Z_\pm \cup \bd_\infty Z_\pm$, we may extend $\psi_{g^{\pm 1}}$ to a semiflow 
$\map{W_\pm \times [0, \infty]}{W_\pm}$ by setting $\psi_{g^{\pm 1}}(z, \infty) \coloneq \gamma_z$ for all $z \in Z_\pm$, 
and putting $\psi_{g^{\pm 1}}(\gamma_z, t) \coloneq \gamma_z$ for all $\gamma_z \in \bd_\infty Z_\pm$ and $t \in [0, \infty]$.
The neighborhoods in $W_\pm$ of boundary points $\gamma_z \in \bd_\infty Z_\pm$ will be sets of the form $\psi_{g^{\pm 1}}(N(z), (t, \infty))$, where $t > 0$, and $N(z)$ is any $Z_\pm$-neighborhood of $z$.
Along with the open sets of $Z_\pm$, these provide the basis for a topology on $W_\pm$ with respect to which this extension of $\psi_g$ is indeed continuous.
We can now define:
\begin{definition}
    The \term{compactified mapping torus} of an end-periodic graph map $g$ is the space
    \[
        W = W_g \coloneq Z \cup \bd_\infty Z_+ \cup \bd_\infty Z_-,
    \]
    whose topology is generated by $W_\pm$-neighborhoods of points in $\bd_\infty Z_\pm$, and open sets in $Z$.
\end{definition}
The sets $\bd_+ W \coloneq \bd_\infty Z_+$ and $\bd_- W \coloneq \bd_\infty Z_-$ are called the \term{positive and negative boundaries} of $W$,
with its total \term{boundary} being their union $\bd W \coloneq \bd_+ W \cup \bd_- W$.
The \term{interior} of $W$ is the difference $W - \bd W$, homeomorphic to the uncompactified mapping torus by definition.

\subsubsection{Invariance under choice of core}
\label{rmk:boundary-points-are-maximal-flowlines}
Although our construction uses a core for $g$ to talk about the sets $Z_+$ and $Z_-$, this is more a convenience than a necessity. 
Without reference to a core, points of $\bd_\infty Z_+$ could instead be described as flowlines of $\psi_g$ 
which are maximal with respect to inclusion and unbounded in the forward direction, with two flowlines considered equivalent if they ever intersect. 
Similarly, elements of $\bd_\infty Z_-$ can be viewed as maximal flowlines of $\psi_g$ unbounded in the reverse direction 
(the fact that backward ends of such flowlines must be distinct means no equivalence relation is necessary).

The equivalence of these definitions follows by observing that the intersection of a maximal flowline with horizontal 1-cells of $Z$ 
can be identified with a a bi-infinite sequence in $\Gamma$ representing the orbit of a point under powers of $g$. 
If the flowline is unbounded, this sequence escapes any compact core for $g$ in either the forward or reverse direction, 
after which it becomes trapped in the positive or negative domain by end-periodicity. 
Therefore, the sets $Z_\pm$ associated to any choice of core eventually meet every unbounded flowline of $\psi_g$.
This is enough to conclude that $W$, like the uncompactified mapping torus $Z$, is uniquely determined by $g$.

\subsubsection{The cell structure of $W_g$}
\label{sec:cell-structure-W}

The only claim of \cref{thm:compactified-mapping-torus} left to prove is that the compactified mapping torus is a finite 2-complex.
We shall do this directly, endowing $W$ with an explicit cell structure derived from that of $Z$.
The combinatorial description developed here will feature prominently in subsequent sections.

Since we can decompose $W = Z_0 \cup W_+ \cup W_-$, where $\cl{Z_0}$ already inherits the structure of compact cell complex from $Z$,
our major task is showing that the closed subsets $\cl{W_+}$ and $\cl{W_-}$ can also be realized as finite 2-complexes.
For concreteness, we shall consider $W_+$ in isolation, 
noting that the analogous argument for $W_-$ is readily obtained by substituting $g$ and $\psi_g$ for $g^{-1}$ and $\psi_{g^{-1}}$.
We can also make the simplifying assumption that our core $\Gamma_0$ is proper, so that, per \cref{sec:positive-and-negative-mapping-tori}, 
$\cl{Z_0}$ meets $\cl{W_+}$ at exactly the finite subcomplex $\cl{B_1} \cup \psi_g(\bd_+ \Gamma_0, [0, 1])$.
The cellulation of $\cl{W_+}$ we build will only need to preserve these cells in order to be reattached to $\cl{Z_0}$.

Considering $\Gamma_+$ within the 1-skeleton of $Z_+$ as usual, we denote the forward orbit of an open edge $e_1^i$ of $B_1$ under the natural semiflow by $E^i \coloneq \psi_g(e_1^i, \halfline)$.
Since $\psi_g$ is injective when restricted to $Z_+$, we may identify $E^i$ with the strip $(0, 1) \times \halfline \subeq \RR^2$ via the map $\psi_g(x, t) \mapsto (\Phi_e, t)$ for all $x \in e_1^i$, $t \ge 0$.
Under this identification, horizontal 1-cells $e_n^i$ correspond to horizontal segments $(0, 1) \times \{n - 1\}$, and unbounded flowlines through $e_1^i$ are taken to vertical rays of the form $\{x\} \times \halfline$.
In particular, this shows $\psigoo(e_1^i) \subeq \bd_+ W$ to be homeomorphic to an open interval.
The closed strip $[0, 1] \times \halfline$, illustrated in \cref{fig:cellstrip}, then corresponds to the 2-complex $F^i$ obtained from $\cl{Z_+}$ by cutting along the forward orbits of $\bd_0 e_1^i$ and $\bd_1 e_1^i$.

\begin{figure}[htpb]
    \centering
    \begin{subfigure}[b]{0.45\textwidth}
        \centering
        \includegraphics[]{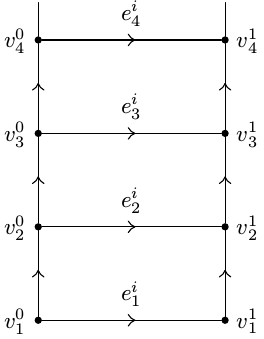}
        \subcaption{$e_1^i$ a subgraph edge.}
        \label{ei}
    \end{subfigure}
    \hfill
    \begin{subfigure}[b]{0.45\textwidth}
        \centering
        \includegraphics[]{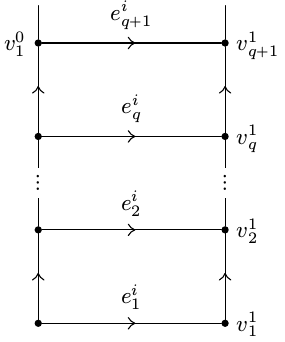}
        \subcaption{$e_1^i$ a joining edge.}
        \label{di}
    \end{subfigure}
    \caption{
        Attaching map for the ideal 2-cell $f_\infty^i$ associated to a subgraph edge (left) and joining edge (right).
        Dotted lines indicate images of horizontal 1-cells $e_{m + 1}^i, e_{m + 2}^i, \dots$ from $F^i$ under the inclusion $F^i \to K^i$; note that these are not 1-cells of $K^i$.
    }
    \label{fig:cellstrip}
\end{figure}

Compactifying the subcomplex $F^i$ along flowlines yields a space $K^i$, which is topologically just the product of $e_1^i$ with a closed interval.
Just as gluing the closed strips $F^1, \dots, F^k$ along their vertical edges recovers $\cl{Z_+}$, one obtains $\cl{W_+}$ by piecing together these compactified strips along their upright boundaries.
This observation motivates our next step, which will be to realize each $K^i$ explicitly as a finite 2-complex.
If we can do this in a way that ensures the vertical sides of adjacent strips glue together by cellular maps, the result will be the desired cellulation of $\cl{W_+}$.

As a preparatory step, we note that each $F^i$ is an infinite tower of 2-cells whose structure can be adapted to a finite cellulation of $K^i$ in roughly the following way.
Fixing a positive integer $m$, called the \term{cutoff} for the strip, we will leave all cells with lower index $\le m$ untouched while interpreting the rays $\psi_g (\bd_0 e_1^i, (m, \infty))$ and $\psi_g (\bd_1 e_1^i, (m, \infty))$ as 1-cells with endpoints at infinity, and $\psi_g (e_1^i, (m, \infty))$ as a 2-cell whose top edge similarly attaches along an ideal horizontal segment.
In order to guarantee neighboring strips attach cellularly, it will be enough to require that both use the same threshold $m$.
Thus, the final cell structure on $\cl{W_+}$ will depend on a separate cutoff for each of its components.
Since the value of this cutoff effectively determines a neighborhood of the corresponding boundary component, we sometimes refer to this as a \term{choice of boundary neighborhoods}.

We now make the above discussion precise.
Observe that for any vertex $v_1^j \in \verts{\cl{B_1}}$, the cell structure of its forward orbit $\psi_g(v_1^j, \halfline)$ in $Z_+$
is that of an infinite ray $\halfline$ subdivided by a 0-cell $v_{n + 1}^j$ at every nonnegative integer point $n$.
We convert this to a cellulation of $\psi_g (v_1^j, [0, \infty])$ by appending an \term{ideal 0-cell} $v_\infty^j$
at $\psigoo(v_1^j)$, and then discarding all vertices $v_n^j$ with $n > m$.
In their place will be a single 1-cell $t_\infty^j$ attaching from $v_m^j$ to $v_\infty^j$.
This realizes the vertical edges of $K^i$ as finite 1-complexes, while preserving the finite cell structure of $\psi_g(e_1^i , [0, m - 1]) \subeq K^i$ inherited from $Z_+$.
The open interval $\psigoo(e^i)$ can now be viewed as an \term{ideal 1-cell} $e_\infty^i$ attaching along the ideal vertices of the left and right edge ($v_\infty^0$ and $v_\infty^1$, say).
Similarly, the open rectangle $\psi_g (e_0^i, (m - 1, \infty))$ becomes the interior of an \term{ideal 2-cell} $f_{\infty}^i$ attached along the perimeter of $[0, 1] \times [m - 1, \infty]$.
The precise spelling of this loop depends on whether $e_1^i$ is a subgraph or joining edge;
in the former case, it is simply $e_m^i \cdot t_\infty^1 \cdot (e_\infty^i)^{-1} \cdot (t_\infty^0)^{-1}$.
If, on the other hand, $e_1^i$ is a joining edge, the lower indices of its initial and terminal vertices are offset by its period $q$,
so the attaching map for $f_\infty^i$ becomes $e_m^i \cdot t_\infty^1 \cdot (e_\infty^i)^{-1} \cdot (t_{m - q}^0 \mathbin{\cdots} t_{m - 1}^0 \cdot t_\infty^0)^{-1}$.
The two situations are presented in \cref{fig:ideal-2-cells}, where we draw $[0, 1] \times [m - 1, \infty]$ scaled by the homeomorphism $(x, y) \mapsto (x, 1 - 2^{m - y})$.

\begin{figure}[htpb]
    \centering
    \begin{subfigure}[b]{0.35\textwidth}
        \centering
        \includegraphics[]{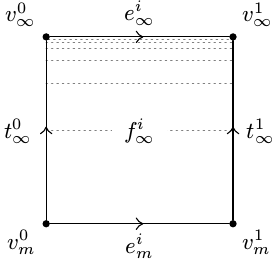}
        \subcaption{ $e_m^i$ a subgraph edge. }
        \label{fig:ideal-2-cells.subgraph}
    \end{subfigure}
    \hfill
    \begin{subfigure}[b]{0.6\textwidth}
        \centering
        \includegraphics[]{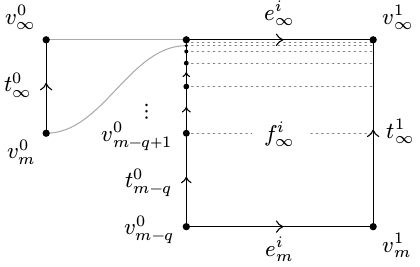}
        \subcaption{$e_m^i$ a joining edge of period $q$. }
        \label{fig:ideal-2-cells.joining}
    \end{subfigure}
    \caption{
        Attaching map for the ideal 2-cell $f_\infty^i$ associated to a subgraph edge (left) and joining edge (right).
        Dotted lines indicate images of horizontal 1-cells $e_{m + 1}^i, e_{m + 2}^i, \dots$ from $F^i$ under the inclusion $F^i \to K^i$;
        note that these are not 1-cells of $K^i$.
    }
    \label{fig:ideal-2-cells}
\end{figure}

Cellulating every $K^i$ in this way and then gluing the collection of strips along vertical boundaries results in a finite cell structure on $\cl{W_+}$.
Since our construction preserves the edges of $\cl{B_1}$ and all vertical 1-cells originating from juncture vertices, there is no obstruction to reattaching $\cl{W_+}$ to $\cl{Z_0}$.
We do the same for $\cl{W_-}$, whose cell structure is also finite, and this completes the proof of \cref{thm:compactified-mapping-torus}.

\begin{remark}
    \label{rmk:cellulation-induced-by-core}
    The cellulation we obtain from this process, while non-canonical, is fully determined our choice of core and boundary neighborhoods.
    The first of these decisions is also the most consequential, since any core, even one that is not proper,
    also determines a set of minimal cutoffs for which the re-gluing of $\cl{W_\pm}$ to $\cl{Z_0}$ remains cellular.
    Thus, a core for $g$ may be said to induce a cell structure on $W$.
\end{remark}

Going forward, we will tacitly assume that every cellulation of a compactified mapping torus we encounter arises from the construction described in \cref{sec:cell-structure-W}.
The following notation will also become standard.

\begin{convention}[Standard notation for mapping tori]
    \label{conv:cellulation-of-W}
    Just as $\map[g]{\Gamma}$ represents a generic end-periodic map, 
    we use $Z$ and $W$ to denote its uncompactified and compactified mapping tori respectively.
    The cell structure of $W$ (when one is fixed) will be constructed relative to a core $\Gamma_0$ for $g$, 
    which induces a block decomposition $\seq[n \in \ZZ]{B_n}$ on $\Gamma$.
    Sometimes we will consider maps denoted $\map[g']{\Gamma'}$; in this case, 
    all corresponding objects will be indicated by primed versions of the same symbols.
\end{convention}

A cellulation of $W$ allows us to describe the structure of its boundary quite explicitly.

\begin{proposition}
    \label{prop:boundary-structure}
    The boundaries $\bd_+ W$ and $\bd_- W$ are finite graphs, and every component of $\bd_\pm W$ is the image of a component of $B_{\pm 1}$ under $\psigoopm$.
    In particular, the component $\psigoopm(B_{\pm 1} \cap U_E)$ is isomorphic to $\cl{B_{\pm 1} \cap U_E} / \mathord{\sim}$,
    where $v \sim g^{\pm\pd{E}}(v)$ for each juncture vertex $v \in \bd_\pm \Gamma_0 \cap \cl{U_E}$.
\end{proposition}

\begin{proof}
    We consider the positive boundary in isolation, the argument for $\bd_- W$ being analogous.
    The restriction of $\psigoo$ to $\Gamma_+$ is a local homeomorphism sending each vertex
    $v_1^j \in \verts{B_1}$ to a 0-cell $v_\infty^j$ in $\bd_+ W$ and each edge $e_1^i \in \edges{B_1}$ to a 1-cell $e_\infty^i$ of $\bd_+ W$.
    Since $\psigoo(x) = \psigoo(y)$ if and only if $x$ and $y$ belong to the same forward orbit under $g$,
    but $B_1 \cap g(B_1) = \emptyset$, it follows these ideal vertices and edges must be distinct.
    \Cref{rmk:boundary-points-are-maximal-flowlines} shows $\restr{\psigoo}{B_1}$ surjects the positive boundary,
    so we conclude $\bd_+ W$ is a finite graph with vertices $\{v_\infty^j : 1 \le j \le \nverts{B_1}\}$ and edges $\{e_\infty^i : 1 \le i \le \nedges{B_1}\}$.

    The attaching map for an ideal edge $e_\infty^i$ is determined by its counterpart $e_1^i$ in $B_1$
    according to $\bd_\ell e_\infty^i = \psigoo(\bd_\ell e_1^i)$ for $\ell \in \{0, 1\}$.
    Since every subgraph edge of $B_1$ is incident to vertices in the same block, the boundary graph contains an isomorphic copy of $\sg{B_1}$.
    On the other hand, if $e_1^i$ is a joining edge, then $\bd_0 e_1^i$ is a juncture vertex, 
    and $v_\infty^j \coloneq \bd_0 e_\infty^i$ corresponds to a vertex $v_1^j \in \verts{B_1}$ satisfying $g^n(\bd_0 e_1^i) = v_1^j$ for some positive $n$.
    By \cref{sec:standard-orientation}, this holds when $n$ matches the period of $e_1^i$.
    Thus, if $e_1^i \in \edges{(B_1 \cap U_E)}$, we have $v_1^j = g^{\pd{E}}(\bd_0 e_1^i) \in g^{\pd{E}}(B_1 \cap U_E) = B_1 \cap U_E$,
    so it follows that $\psigoo(B_{1} \cap U_E)$ comprises a component of $\bd_+ W$ for each $E \in \Ends'_+(\Gamma)$.
\end{proof}

For each leading end $E \in \Ends'_\pm(\Gamma)$, we will let $\Sigma_E$ designate the component $\psigoopm(B_{\pm 1} \cap U_E)$ of $\bd_\pm W$.
The unique correspondence between edges of $\bd_\pm W$ and 1-cells of $B_{\pm1}$ means the distinction between subgraph and joining edges persists into the boundary.
Thus, an element of $\edges{\bd_\pm W}$ will be called an \term{(ideal) subgraph edge} if it is the image under 
$\smash{\psigoopm}$ of an edge from $\sg{B_{\pm 1}}$, and an \term{(ideal) joining edge} otherwise.

\begin{example}
    If $g$ is our example end-periodic map defined in \cref{fig:Gamma}, the positive and negative boundary graphs of its compactified mapping torus are as follows.

    \begin{figure}[htpb]
        \centering
        \hfill
        \begin{subfigure}[c]{0.35\textwidth}
            \centering
            \includegraphics[]{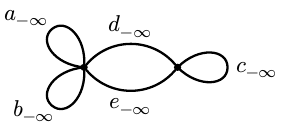}
            \subcaption{$\bd_- W_g$}
        \end{subfigure}
        \hfill
        \begin{subfigure}[c]{0.35\textwidth}
            \centering
            \includegraphics[]{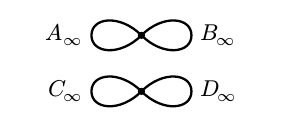}
            \subcaption{$\bd_+ W_g$}
        \end{subfigure}
        \hfill{}
    \end{figure}
\end{example}

\subsection{Graphs in the compactified mapping torus}
The purpose of this section is to introduce graphs in the mapping torus whose dynamics will be of particular interest going forward.
As usual, let $Z$ be the mapping torus of an end-periodic map with compactification $W$.
We say a subset $\gX$ is a \term{graph} or \term{subgraph} in $Z$ or $W$ if it is an embedded locally finite 1-complex (though not necessarily a subcomplex).
The \term{boundary} of a $\gX \subeq W$ is defined to be $\bd \gX \coloneq \gX \cap \bd W$ (likewise, $\bd_\pm \gX \coloneq \gX \cap \bd_\pm W$).
We say that $\gX$ is \term{nicely embedded} in $W$ if $\bd \gX$ is a finite (possibly empty) subset of $\verts{\gX}$ 
consisting of valence-1 vertices meeting $\bd W$ away from $\zskele{\bd W}$.
A graph in $Z$ or the interior of $W$ is said to be \term{transverse} to the natural semiflow
if it has an oriented $[-1, 1]$-bundle neighborhood $\gX \times [-1, 1]$ intersecting flowlines along segments that start at $\gX \times \{-1\}$,
end at $\gX \times \{1\}$, and meet $\gX$ at exactly one point in between.
We consider nicely embedded subgraphs of $W$ to be transverse to the semiflow if their intersection with $\interior W$ is.

Given a graph $\gX$ nicely embedded in $W$ and transverse to the semiflow $\psi_g$, we can consider for each point $x \in \gX$ 
the set $R(x) \coloneq \{t \in (0, \infty) : \psi_g (x, t) \in \gX\}$,
consisting of all positive times at which the forward orbit of $x$ meets the original graph.
If $\min R(x)$ exists for all $x \in \gX$, we define the \term{first return map} of this graph to be the function 
$\gX \to \gX$ sending $x$ to $\psi_g (x, \min R(x))$, its point of first return.

The prototypical example of a graph in $Z$ transverse to the natural semiflow $\psi_g$ is $\Gamma$ itself.
We note that $\Gamma$ has first return map $\psi_g^1 = g$, thus encoding the combinatorial information of the original graph map in the dynamics of its mapping torus.
Upon compactification, $\Gamma$ ceases to be an embedded subgraph; edges in neighborhoods of ends now accumulate onto ideal 1-cells of the boundary.
This leads us to consider certain nicely embedded graphs that are closely related. 
The first of these is called the \emph{interior subgraph}.

\begin{definition}[Interior subgraph]
    Let $W$ be a compactified and cellulated mapping torus.
    The \term{interior subgraph} $\Delta$ associated to this cellulation
    is the subcomplex consisting of all closed horizontal 1-cells from the interior of $W$.
\end{definition}

Clearly, the interior subgraph is a finite subcomplex of $\skele[1]{W}$ that contains every non-ideal vertex.
One can think of $\Delta$ as what remains of $\Gamma$ in the 1-skeleton of $W$ after compactification; 
this is well illustrated in \cref{fig:compactification-4}, where the intersection of $\Delta$ with $\cl{W_+}$ 
corresponds to the union of all closed orange and blue 1-cells.
Since its edges are derived from horizontal 1-cells of $Z$, it is also natural to view $\Delta$ a subset of $\Gamma$, 
and in this regard it is a connected subgraph containing the core $\Gamma_0$.

We note that when $W$ has been cellulated with cutoff $m$ everywhere, the resulting $\Delta$ coincides with $\Gamma_m = \bigcup_{i=-m}^m B_i$, the $m$th enlargement of $\Gamma_0$.
Our next proposition shows such an identity always holds relative to the correct choice of core.

\begin{proposition}
    \label{prop:rebase-interior-subgraph}
    Let $m$ be a positive integer, and suppose $W$ is cellulated according to \cref{conv:cellulation-of-W} with cutoffs $m_E > m$ around each component $\Sigma_E$ of the boundary.
    Denote the associated interior subgraph by $\Delta$.
    Then $\Gamma_0$ is contained in a larger core $\Gamma^*_0$ for $g$ whose induced block decomposition 
    $\seq[n \in \ZZ]{B^*_n}$ is such that $B^*_{n} \homeo B_{n}$ for all $n \ne 0$,
    and $\Delta = \Gamma^*_m$.
\end{proposition}

\begin{proof}
    Each component $V$ of $W_+$ is the forward orbit of a component of $B_1$, 
    which by \cref{prop:Bn-components} can be written $B_1 \cap U_E$ for a leading attracting end $E$.
    Since $V$ is cellulated with cutoff $m_E$, the subset of its 1-skeleton consisting of horizontal 1-cells can be expressed
    \[
        \Delta \cap V = \bigcup_{i = 0}^{m_E - 1} g^i(B_1 \cap U_E),
    \]
    and by letting $E$ vary over all leading attracting ends we obtain a decomposition
    \[
        \Delta \cap W_+ = \bigcup_{E \in \Ends'_+(\Gamma)} \bigcup_{i = 0}^{m_E - 1} g^i(B_1 \cap U_E).
    \]
    Note that all edges of $\Delta$ which meet the closure of ideal 2-cells from $W_+$ are contained in translates of the form $g^{m_E - 1}(B_1 \cap U_E)$.
    Our new core will be defined from the outside in, beginning with
    \[
        B^*_m \coloneq \bigcup_{E \in \Ends'_+(\Gamma)} g^{m_E - 1}(B_1 \cap U_E).
    \]
    This $B^*_m$ is homeomorphic to $B_1$ since each power of $g$ restricts to a homeomorphism of $B_1$, 
    and the images of different components of $B_1$ under various powers of $g$ are all disjoint.
    Since each cutoff $m_E$ is at least $m$, we can then inductively define $B^*_n \coloneq g^{-1}(B^*_{n + 1})$ for all $1 \le n < m$ until we reach $B^*_1$.
    One defines $B^*_{-n}$ similarly, starting with
    \[
        B^*_{-m} \coloneq \bigcup_{E \in \Ends'_-(\Gamma)} g^{-m_E + 1}(B_{-1} \cap U_E)
    \]
    (homeomorphic to $B_{-1}$), and taking $B^*_{-n} \coloneq g(B^*_{-n - 1})$ for all  $1 \le n < m$.
    At this point, $\Gamma^*_0$ is taken to be the remainder of $\Delta$:
    \[
        \Gamma^*_0 \coloneq \Delta - \bigcup_{i = 1}^m (B_i \cup B_{-i}).
    \]
    It is easily verified that $\Gamma^*_0$ represents a well-chosen core for $g$, 
    with all properties of $\seq[n \in \ZZ]{B^*_n}$ in the statement of the proposition satisfied by construction.
\end{proof}

Consider again the compactification pictured in \cref{fig:compactification}.
In going from figure $\labelcref{fig:compactification-3}$ to $\labelcref{fig:compactification-4}$, we discard infinitely many edges of $\Gamma$, 
and are left with a finite graph $\Delta$ that does not meet the boundary.
Now suppose $L \homeo \bd_+ W \times (0, 1]$ is the collar neighborhood of $\bd_+ W$ containing all vertices of $\Gamma_+ - \Delta$ depicted in \cref{fig:pushing-in}.
Since the orange subgraph edges of $\Gamma_+$ are loops whereas the blue joining edges have offset vertices, 
the frontier of $L$ meets $\Gamma_+$ at exactly one point along the interior of a joining edge.
Flowing $L$ into the boundary via a homotopy supported on ideal 2-cells results in a nicely embedded graph $\Lambda$ in $W_+$, containing $\Delta$ and with nonempty boundary.
Generalizing this example motivates our next construction.

\begin{figure}[htpb]
    \centering
    \includegraphics[width=0.72\textwidth]{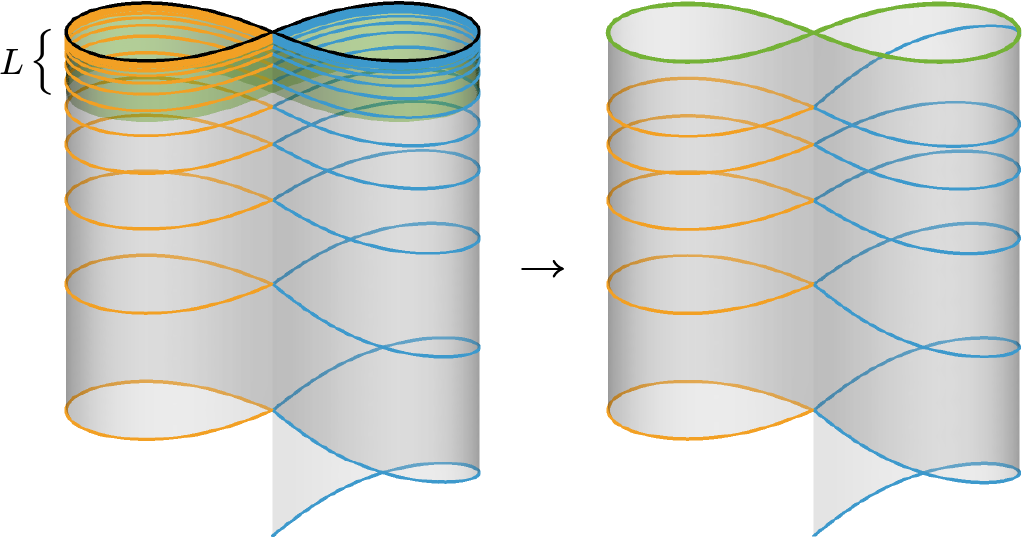}
    \caption{Pushing in a collar neighborhood $L$ of the boundary to obtain a nicely embedded graph.}
    \label{fig:pushing-in}
\end{figure}

Let $e_\infty^i$ be an ideal joining edge of a component $\Sigma \subeq \bd_+ W$ with initial and terminal vertices $v_\infty^0$ and $v_\infty^1$ respectively.
The ideal 2-cell $f_\infty^i$ attaches along $e_\infty^i$ and the horizontal edge $e_m^i$, where $m$ is the cutoff for the cellulation around $\Sigma$.
Letting $q$ denote the period of $e_m^i$, the vertical segment from $v_{m - q}^0 \coloneq \bd_0 e_m^i$ to $v_\infty^0$ along the boundary of $f_\infty^i$
is subdivided by exactly $q$ 0-cells: $v_{m - q + 1}^0, \dots, v_m^0$.
Identifying $f_\infty^i$ with the unit square as in \cref{fig:ideal-2-cells.joining}, 
we now specify for each $j \in \{1, \dots, q\}$ a new 1-cell, $s_+^j$, by drawing an arc from $v_{m - q + j}^0$ to the interior of $e_\infty^i$, 
strictly monotonic in both coordinates, and with closure disjoint from the closure of $s_+^k$ for all other $j \neq k$ (\cref{fig:ideal-subdivision.pos}).
These 1-cells, which we call \term{positive subdividing edges}, have closures nicely embedded in $W$, 
are transverse to the natural semiflow, and collectively partition both $e_\infty^i$ and $f_\infty^i$ into $q + 1$ subcells .
One defines \term{negative subdividing edges} analogously for any joining edge $e_\moo^i$ of $\bd_- W$ by
splitting the ideal 2-cell $f_\moo^i$ along edges $s_-^j$ that attach from 0-cells on the segment between 
$\bd_0 e_{-m}^i$ and $\bd_0 e_\moo^i$ to points on the interior of $e^i_{\moo}$ (\cref{fig:ideal-subdivision.neg}).

\begin{figure}[htpb]
    \centering
    \begin{subfigure}[c]{0.45\textwidth}
        \centering
        \includegraphics[]{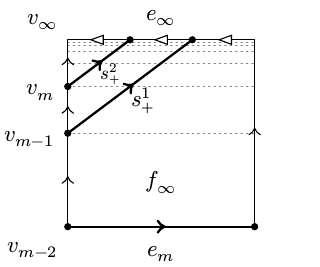}
        \subcaption{Ideal 2-cell of $W_+$.}
        \label{fig:ideal-subdivision.pos}
    \end{subfigure}
    \hfill
    \begin{subfigure}[c]{0.45\textwidth}
        \centering
        \includegraphics[]{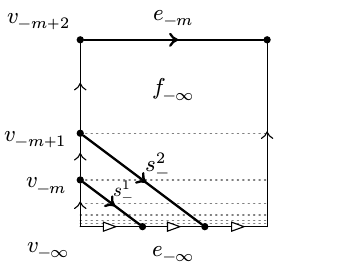}
        \subcaption{Ideal 2-cell of $W_-$.}
        \label{fig:ideal-subdivision.neg}
    \end{subfigure}
    \caption{
        Subdivision for an ideal 2-cell of $W_+$ (left) and $W_-$ (right) which attaches along an ideal joining edge of period 2.
        Edges of $\Lambda$ are drawn in bold, and the derived orientation on $\bd W$ is indicated by open arrowheads.
    }
    \label{fig:subdivided-2-cells}
\end{figure}

\begin{definition}[Principal subgraph of $W$]
    Let $W$ be the compactified and cellulated mapping torus of an end-periodic map $\map[g]{\Gamma}$.
    The \term{principal subgraph} $\Lambda$ associated to this cellulation is defined as the union of 
    its interior subgraph $\Delta$ with the closure of every subdividing edge.
\end{definition}

We observe $\Lambda$ is a finite graph, nicely embedded in $W$ and transverse to the suspension semiflow.
By construction, it intersects every component $\Sigma_E$ of $\bd W$, meeting each joining edge $e_\pmoo \in \edges{\Sigma_E}$ at exactly $\pd{E}$ distinct points.
This graph's interaction with the boundary can be used to impose orientations on ideal joining edges as follows.
For each subdividing edge $s_\pm$ of $\Lambda$ contained in an ideal 2-cell $f_\pmoo$ of $W_\pm$, 
let $f^+$ denote the component of $f_\pmoo - s_\pm$ that meets the forward orbit of $s_\pm$ along $\psi_g$.
The corresponding boundary edge $e_\pmoo$ is endowed with an orientation directed from $e_\pmoo - \cl{f^+}$ to $e_\pmoo \cap \cl{f^+}$.
We call this its \term{derived orientation}, 
in contrast to the \term{standard orientation} it inherits from edges of $J_{\pmn}$, $n \ge 1$ under $\psigoopm$ (recall \cref{sec:standard-orientation}).
From \cref{fig:subdivided-2-cells}, it is clear that every subdividing edge incident to $e_\pmoo$ induces the same derived orientation, 
which does not depend on the cutoff $m$ of the ambient cellulation.

\begin{remark}
    \label{rmk:derived-orientations}
    The derived orientation on each joining edge of the positive boundary is opposite its standard orientation, 
    whereas for joining edges of the negative boundary, the two coincide.
\end{remark}

The way $\Lambda$ sits inside $W$ is recorded by the various data it induces on the boundary.
We refer to a collection of such information as a \emph{decoration}, defined as follows.

\begin{definition}[The decorated boundary]
    \label{rmk:decoration}
    A \term{decoration} of $\bd W$ consists of:
    \begin{itemize}
        \item A \emph{partition} of $\edges{\bd W}$ into subgraph and joining edges.
        \item A \emph{subdivision} of each joining edge in the component $\Sigma_E \subeq \bd W$ by $\pd{E}$ vertices.
        \item A \emph{derived orientation} on joining edges of $\bd W$.
    \end{itemize}
\end{definition}

It should be noted that a decoration of the boundary is really an invariant of the core used in the cellulation of $W$, 
since none of the components above depends on a cutoff.

\begin{example}
    Let $g$ be our running example from \cref{fig:Gamma}.
    The decoration of $\bd W_g$ induced by the core $\Gamma_0$ is as follows.
    All joining edges are highlighted in blue, with subdivisions indicated by valence-2 vertices and derived orientations given by open arrowheads.
    \begin{figure}[htpb]
        \centering
        \hfill
        \begin{subfigure}[c]{0.35\textwidth}
            \centering
            \includegraphics[]{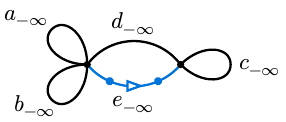}
            \subcaption{$\bd_- W_g$}
        \end{subfigure}
        \hfill
        \begin{subfigure}[c]{0.35\textwidth}
            \centering
            \includegraphics[]{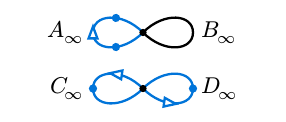}
            \subcaption{$\bd_+ W_g$}
        \end{subfigure}
        \hfill{}
    \end{figure}
\end{example}

\subsection{Coupling}
\label{sec:coupling}

Suppose $\map[g]{\Gamma}$ and $\map[g']{\Gamma'}$ are end-periodic maps.
A \term{decoration preserving map} is a component-wise, coorientation reversing isomorphism $\bd_\pm W \to \bd_\mp W'$ of decorated boundary graphs  which sends subdivided ideal joining edges to one another combinatorially while preserving their derived orientations.
We say that cellulations of $W$ and $W'$ are \term{compatible} if their decorated boundaries can be related by such a map, 
and \term{$h$-compatible} if these decorations are preserved by the automorphism $\map[h]{\bd_\pm W}{\bd_\mp W'}$ in particular.
End-periodic maps will be called ($h$-)compatible if their compactified mapping tori admit ($h$-)compatible cellulations.
On compatible mapping tori, we define an operation called \emph{coupling}.

\begin{definition}[The $h$-couple]
    Let $W$ and $W'$ be compactified mapping tori with compatible cell structures, and $\map[h]{\bd_\pm W}{\bd_\mp W'}$ a decoration preserving map.
    The space
    \[
        M = M(W, W', h) \coloneq W \sqcup_h W'
    \]
    obtained by gluing $W$ and $W'$ along their boundaries via $h$ is called the \term{$h$-couple} of $W$ and $W'$.
\end{definition}

Since the gluing is combinatorial, every $h$-couple is automatically equipped with the structure of a finite 2-complex.
This cell structure depends heavily on the original cellulations of $W$ and $W'$, and is not hard to see that every pair of $h$-compatible cell structures on two mapping tori give rise to a different cellulation of their $h$-couple.
As positive boundary components of $W$ are matched with negative boundary components of $W'$, we can extend their natural semiflows
by first reparameterizing $\psi_g$ and $\psi_{g'}$ to flow through ideal 2-cells in unit time, and then joining flowlines through each identified point on the boundary.
The result is a map $\map[\psi]{M \times \halfline}{M}$ that we call the \term{extended semiflow} on $M$.

\begin{definition}[Principal subgraph of $M$]
    Let $W$ and $W'$ be compactified mapping tori with $h$-compatible cellulations, 
    whose principal subgraphs we denote $\Lambda$ and $\Lambda'$ respectively.
    The \term{principal subgraph} associated to the resulting cellulation of their $h$-couple $M$ is the space
    \[
        \Theta \coloneq \Lambda \sqcup_{\restr{h}{\bd \Lambda}} \Lambda',
    \]
    obtained by gluing $\bd \Lambda$ to $\bd \Lambda'$ by the map $\map[\restr{h}{\bd_\pm \Lambda}]{\bd_\pm \Lambda}{\bd_\mp \Lambda'}$.
\end{definition}

Since $\Lambda$ and $\Lambda'$ are connected, finite, and meet every boundary component of their respective mapping tori, 
$\Theta$ is likewise a finite connected graph in $M$.
The niceness of embeddings $\Lambda \to W$ and $\Lambda' \to W'$ guarantees every positive subdividing edge $s_+$ of $\Lambda$ (resp.\ $\Lambda'$)
is paired with a negative subdividing edge $s_-$ of $\Lambda'$ (resp.\ $\Lambda$) along a single boundary vertex.
It follows that each 0-cell of $\Theta$ descended from a pair of identified boundary vertices has valence 2.
As a matter of convenience, we will omit these points from $\verts{\Theta}$, choosing instead to view the concatenation of $s_+$ and $s_-$
as a single element $s \in \edges{\Theta}$, which we call a \term{subdividing edge} of $\Theta$
(this terminology is unambiguous, since the original $s_+$ and $s_-$ are not themselves elements of $\edges{\Theta}$).
The union of subdividing edges which attach a vertex of $\Delta \cap \Gamma_+$ to a vertex of $\Delta' \cap \Gamma'_-$ will be denoted $S_+$,
with $S_-$ likewise designating the union of edges which connect $\Delta' \cap \Gamma'_+$ to $\Delta \cap \Gamma_-$.
Where relevant, we adopt the convention that edges of $S_+$ are oriented with initial vertices in $\Delta$ and terminal vertices in $\Delta'$, the opposite holding for edges of $S_-$.
What follows is a decomposition of $\Theta$ as the disjoint union $\Delta \cup \Delta' \cup S_+ \cup S_-$.

\begin{proposition} \label{prop:Theta}
    Let $\map[g]{\Gamma}$ and $\map[g']{\Gamma'}$ be 
    $h$-compatible end-periodic maps with compactified mapping tori $W$ and $W'$, 
    and suppose $\Theta$ is a principal subgraph for some cellulation of their $h$-couple $M$.
    Then the following hold:
    \begin{enumprop}
        \item \label{prop:Theta.transverse} $\Theta$ is transverse to the extended semiflow $\psi$, 
        \item \label{prop:Theta.f-exists} $\Theta$ has a well-defined first return map $\map[f]{\Theta}$, and 
        \item \label{prop:Theta.M-homeo-to-mapping-torus} $M$ is homeomorphic to the mapping torus $Z_f$.
    \end{enumprop}
\end{proposition}

\begin{proof}
    Fix $h$-compatible cellulations of $W$ and $W'$, and let $\Delta$ and $\Lambda$ be the interior and principal subgraphs for $W$, 
    with $\Delta'$ and $\Lambda'$ defined similarly for $W'$.
    We have $\Theta = \Delta \cup \Delta' \cup S_+ \cup S'_-$.
    Since $\Lambda$ and $\Lambda'$ are transverse to the natural semiflows of their enclosing mapping tori, 
    $\Theta$ is transverse to the extended semiflow everywhere except possibly the identified boundary vertices $\bd \Lambda \sim \bd \Lambda'$
    Thus, part \refitem{prop:Theta.transverse} only needs to be verified for edges of $S_+ \cup S_-$, which can be done in the course of proving part \refitem{prop:Theta.f-exists}.
    Since the roles of $g$ and $g'$ are not distinguished, it will be enough to argue $f$ is well-defined
    on $\Delta \cup S_+$; the same follows automatically for $\Delta' \cup S_-$ by exchanging $g$ for $g'$.

    For each component $\Sigma$ of $\bd_\pm W$, let $\Sigma' \coloneq h(\Sigma)$ denote its counterpart in $\bd_\pm W'$.
    The union of all closed ideal 2-cells from $W$ and $W'$ that attach along $\Sigma \sim \Sigma'$ specify a closed neighborhood of this component in $M$, which we denote $N(\Sigma)$.
    Deleting the interior of $N(\Sigma)$ and doing likewise for every other component of $\bd W$ separates $M$ 
    into components $\trunc{W}$ and $\trunc{W'}$, which can be identified with the subcomplexes of $W$ and $W'$ consisting of cells whose closure does not contain an ideal vertex.
    By definition, $\Delta$ is exactly the union of closed horizontal 1-cells in $\trunc{W}$.
    Thus, the edges of $\Delta$ which flow into 2-cells of $\trunc{W}$ have first return map given by $\psi^1 = \psi_g^1 = g$, 
    and those which do not must represent the bottom edge of an ideal 2-cell in $W_+$.


    The only edges of $\Delta \cup S_+$ left to consider are those which intersect a neighborhood $N(\Sigma)$ of some boundary component $\Sigma \subeq \bd_+ W$.
    Each edge $e_\infty$ of $\Sigma$ is paired with an edge $e'_\moo$ of $\Sigma'$ under $h$, and supposing for convenience that both $W$ and $W'$ have been cellulated with cutoff $m$, we take $f_\infty \in \faces{W}$ and $f'_\moo \in \faces{(W')}$ to be the ideal 2-cells that attach along $e_\infty$ and $e_m$ and along $e'_\moo$ and $e'_{-m}$ respectively.
    We define the \term{ideal neighborhood} of $e_\infty$ to be its 2-sided neighborhood in $N(\Sigma)$ given by $N(e_\infty) \coloneq e_m \cup f_\infty \cup e_\infty \cup f'_\moo \cup e'_{-m}$.

    If $e_\infty$ and $e'_\moo$ are both subgraph edges, then $\Theta$ meets $N(e_\infty)$ exactly twice at the edges $e_m$ and $e'_{-m}$.
    Identifying 2-cells with squares in $\RR^2$ as usual, the identification of $e_\infty$ with $e'_\moo$ glues the top of $f_\infty$ to the bottom of $f'_\moo$, so that $N(e_\infty)$ becomes a rectangle with bottom edge $e_m$ and top edge $e'_{-m}$ (\cref{fig:ideal-neighborhood.subgraph}).
    Trajectories under the extended semiflow move upward along vertical segments, and it is clear that $f(e_m) = e'_{-m}$.

    On the other hand, if $e_\infty$ and $e'_\moo$ are both joining edges, then $\Theta \cap N(e_\infty)$ consists of $e_m$, $e'_{-m}$, and $q$ subdividing edges $s^1, \dots, s^q$, where $q$ is the period of $e_m$.
    Again, $N(e_\infty)$ is homeomorphic to a rectangle with bottom edge $e_m$, top edge $e'_{-m}$, and $\psi$ flowing points upward along vertical segments (\cref{fig:ideal-neighborhood.joining}).
    The requirement that $h$ match the derived orientations on $e_\infty$ and $e'_\moo$ ensures that
    each subdividing edge $s^i$ begins on one vertical edge of $N(e_\infty)$,
    ends on the other, and remains transverse to $\psi$ at every point in between,
    justifying part \refitem{prop:Theta.transverse} of the proposition.
    Furthermore, we see that $s^i$ meets every flowline through $N(e_\infty)$ exactly once,
    so $f$ is well-defined and maps $e_m \mapsto s^1 \mapsto \cdots \mapsto s^q \mapsto (e'_{-m})^{-1}$ as oriented edges.

    This proves $f$ is well-defined on $\Delta \cup S_+$, and the same follows for $\Delta' \cup S_-$ by exchanging $g$ for $g'$, completing part \refitem{prop:Theta.f-exists}.
    For item \refitem{prop:Theta.M-homeo-to-mapping-torus}, we simply note that every point $x$ of $M$ is either contained in the neighborhood $N(\Sigma)$ of a boundary component, 
    or one of the subcomplexes $\trunc{W}$ and $\trunc{(W')}$.
    The previous arguments now show any such $x$ can be reached by flowing forward a unique point of $\Theta$, and this completes the proof.
\end{proof}

\begin{figure}[p]
    \centering
    \begin{subfigure}[c]{0.99\textwidth}
        \centering
        \includegraphics[width=\textwidth]{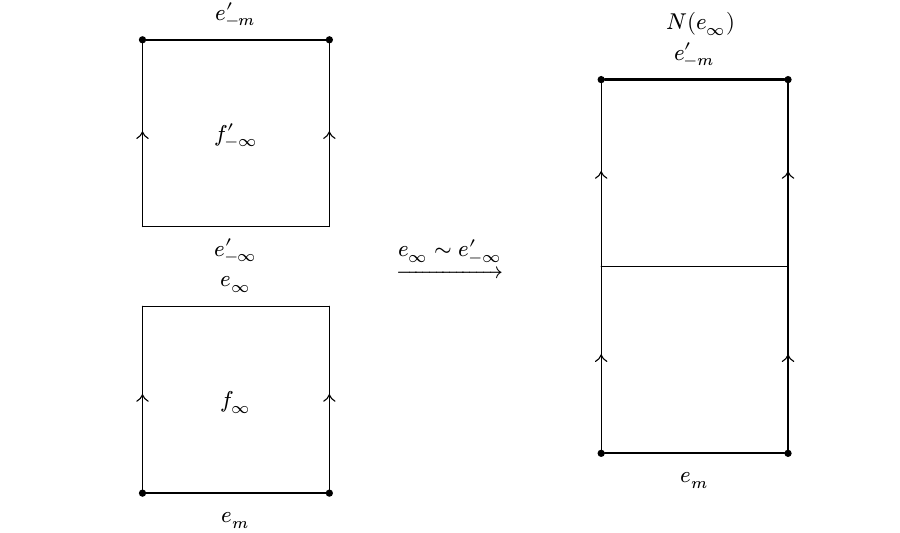}
        \subcaption{Case 1: $e_\infty$ an ideal subgraph edge.}
        \label{fig:ideal-neighborhood.subgraph}
    \end{subfigure}
    \begin{subfigure}[c]{0.99\textwidth}
        \centering
        \includegraphics[width=\textwidth]{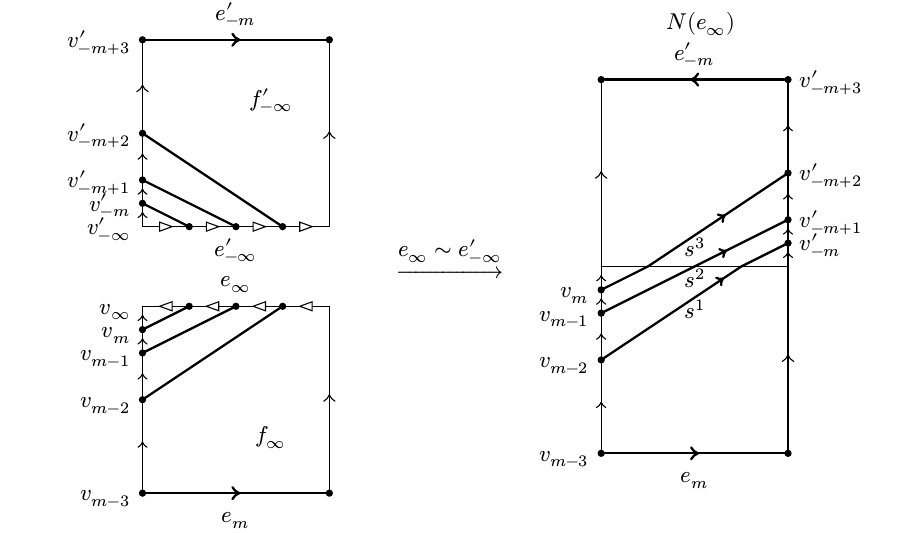}
        \subcaption{        Case 2: $e_\infty$ an ideal joining edge.
            The depicted $e_m$ has period 3, and open arrowheads indicate derived orientations on ideal edges, which must be matched upon gluing.}
        \label{fig:ideal-neighborhood.joining}
    \end{subfigure}
    \caption{
        Gluing ideal 2-cells to form the two-sided neighborhood $N(e_\infty)$ in $M$.
        Bold is used for edges of $\Lambda$ and $\Lambda'$ on the left, which become edges of $\Theta$ on the right.
        \label{fig:ideal-neighborhood}
    }
\end{figure}

From this proof, one easily observes that $\psi$ flows each vertex of $\Theta$ forward to another vertex in $\Theta$; thus, $f$ is a graph map.
Assuming $W$ and $W'$ are cellulated with uniform cutoff $m$ around each boundary component, 
we write their interior subgraphs as $\Delta = \bigcup_{i = - m}^m B_i$ and $\Delta' = \bigcup_{i = - m}^m B'_i$
in order to make the following observations.

\begin{proposition}
    \label{prop:properties-of-Theta}

    Suppose $W$ and $W'$ are cellulated as above.
    Then first return map $f$ of $\Theta = \Delta \cup \Delta' \cup S_+ \cup S_-$ has the following properties.
    \begin{enumprop}
        \item $f$ restricts to match $g$ on $\Delta - B_m$, and $g'$ on $\Delta' - B'_m$. \label{prop:properties-of-Theta.f-matches-g}
        \item $f$ induces isomorphisms $\cl{B_m \cup S_+} \to \cl{B'_{-m} \cup S_+}$ and $\cl{B'_m \cup S_-} \to \cl{B_{-m} \cup S_-}$. \label{prop:properties-of-Theta.isomorphisms-of-lambda}
    \end{enumprop}
\end{proposition}

\begin{proof}
    The first claim was proved in the previous proposition, and we record it here for future reference.
    For \refitem{prop:properties-of-Theta.isomorphisms-of-lambda}, we must show that $f$ restricts to an isomorphism from 
    $\Lambda_1 \coloneq \cl{B_m \cup S_+}$ to $\Lambda_2 \coloneq \cl{B'_{-m} \cup S_+}$.
    These being finite graphs, it will suffice to show the continuous function $\restr{f}{\Lambda_1}$ is an injection onto $\Lambda_2$.
    Since $B_m \homeo B_1$, \cref{prop:boundary-structure} implies $\restr{\psigoo}{B_m}$ gives rise to a bijection 
    $\edges{B_m} \leftrightarrow \edges{(\bd_+ W)}$, and it holds that $\edges{(B'_{-m})} \leftrightarrow \edges{(\bd_- W')}$ by the same token.
    What's more, $h$ induces a correspondence $\edges{(\bd_+ W)} \leftrightarrow \edges{(\bd_- W')}$, so that edges of $B_m$ correspond uniquely to edges of $B'_{-m}$ by transitivity.
    In particular, the subgraphs $\sg{B_m}$ and $\sg{B'_{-m}}$, which are isomorphic to subgraphs of the associated boundaries, must be isomorphic to one another.

    We first prove $f(\Lambda_1) = \Lambda_2$.
    Let $e_{\infty}$ be an edge of $\bd_+ W$, paired with $e'_\moo$ under $h$.
    As noted in the proof of \cref{prop:Theta}, its ideal neighborhood $N(e_\infty)$ has the cell structure of a rectangle whose top and bottom segments 
    relative to the upward oriented semiflow $\psi$ are edges of $B_m$ and $B'_{-m}$ respectively; 
    the reader should again refer to \cref{fig:ideal-neighborhood} for depictions. 
    We note that each 1-cell of $B_m$, $S_+$, and $B'_{-m}$ is contained in exactly one such ideal neighborhood.
    The first return map $f$ takes each labeled edge of $\Theta$ in this figure to the one directly above it.
    Thus, 1-cells of $\sg{B_m}$ map directly to 1-cells of $\sg{B'_{-m}}$, and edges of $J_m$ are sent to edges of $S_+$.
    At the same time, a subset of $S_+$ maps onto $J'_{-m}$, while the elements of $\edges{S_+}$ not sent to $B'_{-m}$ flow back into $S_+$.
    This implies $f(\Lambda_1 - \skele[0]{\Lambda_1}) = \Lambda_2 - \skele[0]{\Lambda_2}$, so the claim follows by the continuity of $f$ and discreteness of $\skele[0]{\Theta}$.

    All that remains is to verify injectivity of $\restr{f}{\Lambda_1}$ at vertices, a task which reduces to showing $\nverts{\Lambda_1} = \nverts{\Lambda_2},$ because both graphs are finite.
    As $\cl{S_+}$ is a subgraph of both $\Lambda_1$ and $\Lambda_2$, we can proceed by counting the 0-cells of $\cl{B_m} - \cl{S_+}$ and $\cl{B'_{-m}} - \cl{S_+}$.
    Observe that by \cref{fig:ideal-neighborhood.joining}, vertices of $\cl{B_m} \cap \cl{S_+}$ correspond to images under $f$ of initial vertices from edges of $J_m$.
    Similarly, each vertex of $\cl{B'_{-m}} \cap \cl{S_+}$ flows to the initial vertex of an edge in $J'_{-m}$.
    Thus
    \[
        \nverts{\cl{B_m} - \cl{S_+}} = \nverts{\sg{B_m}} = \nverts{\sg{B'_{-m}}} = \nverts{\cl{B'_{-m}} - \cl{S_+}},
    \]
    where the middle equality is a consequence of the isomorphism between $\sg{B_m}$ and $\sg{B'_{-m}}$, and the conclusion follows.
\end{proof}

\section{End-periodic homotopy equivalences}
Recall that a graph map $\map[f]{\gX}{\gY}$ is considered a \term{proper homotopy equivalence} (PHE) if it is proper, and there 
exists a proper map $\map[f']{\gY}{\gX}$ such that $f f'$ and $f' f$ are properly homotopic to the identity.
The goal of this section is to show that end-periodic homotopy equivalences automatically satisfy this property, 
and therefore represent elements of the mapping class group.
Since end-periodic maps are proper by definition, this will be a consequence of the following theorem.

\homotopyInverseTheorem

The first step toward proving \cref{thm:homotopy-inverse} is to establish a generating set for $\pi_1(\Gamma)$ using the following standard construction.
For an arbitrary subset $U \subeq \Gamma$, we define a \term{maximal tree} for $U$ to be a connected, acyclic subgraph of $U$ containing $\zskele{U}$, 
and a \term{maximal forest} for $U$ to be a subgraph of $U$ that meets each component in a maximal tree.
Given a maximal tree $T$ for $\Gamma$, and vertices $v_1, v_2 \in \verts{\Gamma}$, 
there is always exactly one reduced edge path in $T$ from $v_1$ to $v_2$, which we denote $\path{v_1, v_2}_T$.
Fixing a basepoint $\pt \in \verts{\Gamma}$, every maximal tree $T$ for $\Gamma$ 
corresponds to a free basis of $\pi_1(\Gamma, \pt)$ consisting of homotopy classes 
of loops $\path{\pt, \bd_0e}_T \cdot e \cdot \path{\bd_1e, \pt}_T$ 
taken as $e$ ranges over all oriented edges in $\Gamma - T$.

To make the action of $g$ on $\pi_1(\Gamma)$ explicit, we will be interested in maximal trees with the following technical property.

\begin{definition}[End-invariant maximal trees]
    Suppose $\map[g]{\Gamma}$ is an end-periodic graph map.
    A maximal tree $T$ for $\Gamma$ is \term{end invariant} under $g$ if there exists a well-chosen core $\Gamma_0$ such that 
    $g^{\pm 1}(\path{v_1, v_2}_T) = \path{g^{\pm 1}(v_1), g^{\pm 1}(v_2)}_T$ for any pair of vertices $v_1, v_2$ from the same 
    component of $\Gamma - \Gamma_0$ (that is, from the same nesting neighborhood).
\end{definition}

\begin{lemma}
    \label{lemma:maximal-tree}
    Let $\Gamma$ be an infinite connected graph and $\map[g]{\Gamma}$ an end-periodic graph map.
    There exists a maximal tree $T$ for $\Gamma$ which is end invariant.
\end{lemma}

\begin{proof}
    Let $\Gamma_0$ be a core for $g$, enlarged once if necessary to ensure $g(B_{-1})$ and $g^{-1}(B_1)$ are disjoint.
    Take $F_1$ and $F_{-1}$ to be maximal forests for the largest subgraphs of $B_1$ and $B_{-1}$ respectively.
    The components of $F_{\pm 1}$ are maximal trees for components of $\sg{B_{\pm 1}}$, 
    which by \cref{prop:Bn-components} must be of the form $\sg{B_{\pm 1}} \cap U_E$, where $E$ is a leading end.
    For every such $E$, we select a \term{critical edge} $e_E$ from $J_{\pm 1} \cap U_E$ 
    (which is nonempty by \cref{rmk:at-least-one-joining-edge-per-component}) 
    whose initial vertex is in $\Gamma_0$ and whose terminal vertex is in $F_{\pm 1} \cap U_E$.
    With $D_{\pm 1}$ designating the union of critical edges over all components of $F_{\pm 1}$, 
    we see that $g$ sends the closure of $F_{-1} \cup D_{-1}$ into $\Gamma_0$ homeomorphically,
    while $g^{-1}$ does likewise for the closure of $F_1 \cup D_1$.
    Since $g^{-1}(F_1 \cup D_1)$ and $g(F_{-1} \cup D_{-1})$ are disjoint, 
    their union meets $\Gamma_0$ in an acyclic subgraph, and, letting $T_0$ be a maximal tree for $\Gamma_0$ containing this union, 
    it follows at once that
    \[
        T_1 \coloneq T_0 \cup (F_1 \cup D_1) \cup (F_{-1} \cup D_{-1})
    \]
    is a maximal tree for $\Gamma_1 = \Gamma_0 \cup B_1 \cup B_{-1}$.
    For all $n > 1$, we now inductively set $F_{\pm(n + 1)} \coloneq g^{\pm 1}(F_\pmn)$ 
    and $D_{\pm(n + 1)} \coloneq g^{\pm 1}(D_\pmn)$ in order to define
    \[
        T_{n + 1} \coloneq T_n \cup (F_{n + 1} \cup D_{n + 1}) \cup (F_{-n - 1} \cup D_{-n - 1}).
    \]
    Because $g$ is end-periodic, each $F_\pmn$ is a maximal forest for the largest subgraph of $B_\pmn$.
    Moreover, $D_\pmn$ contains exactly one edge connecting $\Gamma_{n - 1}$ to any given component of $F_\pmn$.
    Thus, $T_n$ is a maximal tree for $\Gamma_n$, and the increasing union 
    $T \coloneq \bigcup_{n \ge 0} T_n$ represents a maximal tree for $\Gamma$.

    To show $T$ is end invariant, we first observe that components of $T - \Gamma_0$ represent maximal trees for nesting neighborhoods.
    Indeed, a path between any two vertices in the same component of $\Gamma - \Gamma_0$ 
    (for concreteness say $v_1 \in B_1 \cap U_E$ and $v_2 \in B_n \cap U_E$) 
    can be obtained by bridging the sequence of unique edges 
    $D_2 \cap U_E, \dots, D_n \cap U_E$ with subpaths from $F_1 \cap U_E, \dots, F_n \cap U_E$.
    Furthermore, our careful selection of $T_0$ and subsequent construction of $T_n$ ensures that
    $g(F_\pmn \cup D_\pmn) \subeq F_{\pmn + 1} \cup D_{\pmn + 1} \subeq T$
    and
    $g^{-1}(F_\pmn \cup D_\pmn) \subeq F_{\pmn - 1} \cup D_{\pmn - 1} \subeq T$
    for all $n \ge 1$.
    Thus, $g^{\pm 1}(T - \Gamma_0) \subeq T$.
    End-invariance is an immediate consequence: 
    if $v_1$ and $v_2$ lie in the same component of $\Gamma - \Gamma_0$, 
    then the path $\path{v_1, v_2}_T$ remains in $T - \Gamma_0$.
    Thus its homeomorphic image under $g$ is a reduced path in $T$ between $g(v_1)$ to $g(v_2)$, 
    equal to $\path{g(v_1), g(v_2)}_T$ by uniqueness.
    The same holds for the image of $\path{v_1, v_2}_T$ under $g^{-1}$.
\end{proof}

Armed with an end-invariant maximal tree, we can now prove our theorem.

\begin{proof}[Proof of \Cref{thm:homotopy-inverse}]
    Fix a base vertex $\pt$ in $\Gamma$, and if necessary, conjugate $g$ by a compactly supported homotopy equivalence $\Gamma \to \Gamma$ to ensure $g(\pt) = \pt$.
    Since the conjugated function matches the original outside some compact set, it is still end-periodic.
    Moreover, the function $g'$ we build to match the new $g^{-1}$ on neighborhoods of ends will also match the inverse of the original on a suitable domain.

    As a homotopy equivalence fixing the basepoint, $g$ induces an automorphism $g_*$ of $\pi_1(\Gamma, \pt)$ whose inverse we denote $\gsinv$.
    Construct as in \cref{lemma:maximal-tree} a maximal tree $T$ that is end-invariant with respect to a core $\Gamma_0$ for $g$, and, 
    letting $\path{u, v}$ denote the path from $u$ to $v$ in $T$, 
    put $\gamma_e \coloneq \path{\pt, \bd_0e} \cdot e \cdot \path{\bd_1e, \pt}$ for each oriented edge $e \in \edges{(\Gamma - T)}$.
    Because $T$ is maximal, the collection of homotopy classes $\{[\gamma_e] : e \in \cells[1]{\Gamma - T}\}$ is a free basis for $\pi_1(\Gamma, \pt)$.

    Recall from the proof of \cref{lemma:maximal-tree} that for each end $E$ of $\Gamma$, the subtree $T \cap \sg{U_E}$ connects to $T_0 = T \cap \Gamma_0$ 
    via a unique critical edge $e_E$ in $J_{\pm 1} \cap U_E$.
    Thus, for any $e \in \edges{(\sg{U_E} - T)}$, we have
    \[
        \gamma_e \homot \path{\pt, \bd_0e_E} \cdot \path{\bd_0e_E, \bd_0e} \cdot e \cdot \path{\bd_1e, \bd_0e_E} \cdot \path{\bd_0e_E, \pt}.
    \]
    Since all but the first and last paths are contained in $\sg{U_E}$, and $T$ is end invariant, it follows
    \begin{align*}
        g(\gamma_e) & \homot g(\path{\pt, \bd_0e_E}) \cdot \path{g(\bd_0e_E), g(\bd_0e)} \cdot g(e) \cdot \path{g(\bd_1e), g(\bd_0e_E)} \cdot g(\path{\bd_0e_E, \pt}) \\
                    & \homot g(\path{\pt, \bd_0e_E}) \cdot \path{g(\bd_0e_E), \pt} \cdot \gamma_{g(e)} \cdot \path{\pt, g(\bd_0e_E)} \cdot g(\path{\bd_0e_E, \pt}).
    \end{align*}
    Using $\sigma_{g(E)}$ to denote the loop $g(\path{\pt, \bd_0e_E}) \cdot \path{g(\bd_0e_E), \pt}$ as in \cref{fig:maxtree}, this shows $g_*([\gamma_e]) = [\sigma_{g(E)}] [\gamma_{g(e)}] [\sigma_{g(E)}]^{-1}$,
    and consequently $\gsinv([\gamma_{g(e)}]) = [\tau_{g(E)}]^{-1} [\gamma_e] [\tau_{g(E)}]$, where $\tau_{g(E)}$ is a representative of $\gsinv([\sigma_{g(E)}])$.
    Reindexing, we obtain
    \[
        \gsinv([\gamma_e]) = [\tau_E^{-1} \cdot \gamma_{g^{-1}(e)} \cdot \tau_E]
    \]
    for all $e \in \edges{(\sg{U_E} - T)}. $

    \begin{figure}[htpb]
        \centering
        \includegraphics[width=0.8\textwidth]{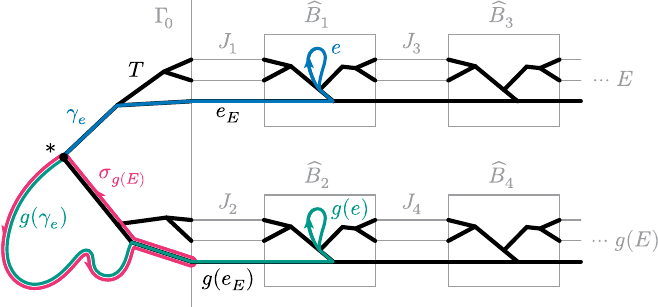}
        \caption{Schematic for the proof of \cref{thm:homotopy-inverse}.}
        \label{fig:maxtree}
    \end{figure}

    We are now ready to define the homotopy inverse $\map[g']{\Gamma'}$.
    Per the theorem statement, we want $g'$ to restrict to $g^{-1}$ in a neighborhood of each end, so we begin by declaring that $g'(x) = g^{-1}(x)$ for all points $x \in  \sg{\Gamma_+} \cup \sg{\Gamma_-}$.
    Inside the core, $g'$ will fix the subtree $T_0$.
    The combination of these two directives means $g'$ acts as follows on the vertices of $\Gamma$:
    \[
        g'(v) = \begin{cases}
            v         & \text{if } v \in \verts{\Gamma_0}             \\
            g^{-1}(v) & \text{if } v \in \verts{(\Gamma - \Gamma_0)}.
        \end{cases}
    \]

    We still need to define how $g'$ maps points of $\Gamma - (\sg{\Gamma_+} \cup \sg{\Gamma_-} \cup T_0)$, but since this is a union of open 1-cells, and the only requirement left to satisfy is that $g'_* = \gsinv$, it will be enough specify the image of each remaining edge $e$ as a path from $g'(\bd_0 e)$ to $g'(\bd_1 e)$.
    Any such edge lies in either $\Gamma_0 - T$ or $J_1 \cup J_{-1}$, and those outside the core must be contained in some nesting neighborhood $U_E$.
    Letting $\eta_e$ denote a reduced representative of $\gsinv([\gamma_e])$ for any edge $e$ outside $T$, we set $g'$ as follows:
    \[
        g'(e) = \begin{cases}
            \path{\bd_0e, \pt} \cdot \eta_e \cdot \path{\pt, \bd_1e}                            & \text{if } e \in \edges{(\Gamma_0 - T)}            \\
            \path{\bd_0e_E, \pt} \cdot \tau_E^\inv \cdot \path{\pt, g^{-1}(\bd_1 e_E)}          & \text{if } e = e_E \text{, a critical edge}        \\
            \path{\bd_0e, \pt} \cdot \eta_e \cdot \tau_E^\inv \cdot \path{\pt, g^{-1}(\bd_1 e)} & \text{if } e \in \edges{(J_{\pm 1} \cap U_E - T)}.
        \end{cases}
    \]

    Having fully determined $g'$ up to homotopy, it remains to verify that $g'_* = \gsinv$.
    We will go about this directly by considering the homotopy class of $g'(\gamma_e)$ for each type of edge $e \in \edges{(\Gamma - T)}$.
    There are three cases to consider, depending on whether $e$ lies in $\Gamma_0$, $J_{\pm 1}$, or $\sg{\Gamma_\pm}$.

    If $e \in \edges{(\Gamma_0 - T)}$, then the first and final subpaths of $\gamma_e = \path{\pt, \bd_0e} \cdot e \cdot \path{\bd_1 e, \pt}$ lie in $T_0$, and are therefore fixed by $g'$.
    It follows that
    \[
        g'(\gamma_e) \homot \path{\pt, \bd_0e} \cdot \path{\bd_0e, \pt} \cdot \eta_e \cdot \path{\pt, \bd_1e} \cdot \path{\bd_1 e, \pt} \homot \eta_e.
    \]

    Alternatively, if $e$ is a non-critical edge contained in $J_{\pm 1} \cap U_E$, we have
    \[
        \gamma_e \homot \path{\pt, \bd_0e} \cdot e \cdot \path{\bd_1e, \bd_1e_E} \cdot e_E^{-1} \cdot \path{\bd_0e_E, \pt},
    \]
    where $\path{\pt, \bd_0e}$ and $\path{\bd_0e_E, \pt}$ are in $T_0$, while $\path{\bd_1e, \bd_1e_E}$ remains in $\sg{U_E}$.
    Because $g'$ fixes the first two paths while sending the last one to $g^{-1}(\path{\bd_1e, \bd_1e_E}) = \path{g^{-1}(\bd_1 e), g^{-1}(\bd_1 e_E)}$,
    the expansion of $g'(\gamma_e)$ is also homotopic to $\eta_e$ after cancellation:
    \begin{align*}
        g'(\gamma_e) \homot \;
               & \path{\pt, \bd_0e} \cdot \left(\path{\bd_0e, \pt} \cdot \eta_e \cdot \tau_E^{-1} \cdot \path{\pt, g^{-1}(\bd_1 e)}\right) \cdot \path{g^{-1}(\bd_1 e), g^{-1}(\bd_1 e_E)} \\
               & \hphantom{\path{\pt, \bd_0e}} \cdot \left(\path{g^{-1}(\bd_1 e_E), \pt} \cdot \tau_E \cdot \path{\pt, \bd_0e_E}\right) \cdot \path{\bd_0e_E, \pt}                         \\
        \homot \; & \eta_e.
    \end{align*}

    Finally, we can write any $e \in \edges{(\sg{\Gamma_\pm} - T)}$ contained in the neighborhood $\sg{U_E}$ as
    \[
        \gamma_e \homot \path{\pt, \bd_0e_E} \cdot e_E \cdot \alpha_e \cdot e_E^\inv \cdot \path{\bd_0e_E, \pt},
    \]
    where $\alpha_e$ designates the loop $\path{\bd_1e_E, \bd_0e} \cdot e \cdot \path{\bd_1e, \bd_1e_E}$.
    Since $\alpha_e$ is a path in $\sg{U_E}$, we see
    \[
        g'(\alpha_e) = g^\inv (\alpha_e) \homot \path{g^{-1}(\bd_1 e_E), g^{-1}(\bd_0 e)} \cdot g^{-1}(e) \cdot \path{g^{-1}(\bd_1 e), g^{-1}(\bd_1 e_E)},
    \]
    and consequently,
    \[
        \path{\pt, g^\inv(\bd_1 e_E)} \cdot g'(\alpha_e) \cdot \path{g^{-1}(\bd_1 e_E), \pt} \homot \gamma_{g^{-1}(e)}.
    \]
    Thus,
    \begin{align*}
        g'(\gamma_e) \homot \; & \path{\pt, \bd_0e_E} \cdot \left(\path{\bd_0e_E, \pt} \cdot \tau_E^{-1} \cdot \path{\pt, g^{-1}(\bd_1 e_E)}\right) \cdot g'(\alpha_e)               \\
                               & \hphantom{\path{\pt, \bd_0e_E}} \cdot \left(\path{g^{-1}(\bd_1 e_E), \pt} \cdot \tau_E \cdot \path{\pt, \bd_0e_E}\right) \cdot \path{\bd_0e_E, \pt} \\
        \homot \;              & \tau_E^\inv \cdot \gamma_{g^{-1}(e)} \cdot \tau_E,
    \end{align*}
    and since this loop represents $\gsinv([\gamma_e])$, the proof is complete.
\end{proof}

\subsection{End-periodic mapping classes and embeddings of mapping tori}
If $\map[g, g']{\Gamma}$ are as in the statement of \cref{thm:homotopy-inverse},
then $g' g$ and $g g'$ must be properly homotopic to the identity, since they match it everywhere outside a compact set.
Thus, every end-periodic homotopy equivalence is a PHE of its domain.
In light of this, we call $\zeta \in \Maps(\Gamma)$ an \term{end-periodic mapping class} if it has an end-periodic representative.
This section is devoted to proving the following theorem about end-periodic mapping classes.

\embeddingTheorem

\begin{proof}
    The proof relies on three key lemmas, to be established formally in the sections to follow.
    The first of these, \cref{lemma:coupling-possible}, says that we can find compatible end-periodic homotopy equivalences $\map[g, g']{\Gamma}$ 
    such that $g$ is a representative of $\zeta$.
    As such, we can glue their compactified mapping tori $W$ and $W'$ together along a decoration preserving map $\map[h]{\bd_\pm W}{\bd_\mp W'}$ to form the $h$-couple $M$.

    By \cref{prop:Theta}, we already know that $M$ can be realized as a finite mapping torus in many ways, namely as $Z_f$, 
    where $f$ is the first return map of any principal subgraph.
    The next ingredient in the construction is \cref{lemma:f-homotopy-equivalence}, 
    which confirms the existence of a principal subgraph $\Theta$ for $M$ whose first return map is a homotopy equivalence.
    Thus, the natural inclusion $Z \to W \to M$ is a flow preserving embedding of mapping tori.

    It remains to show $\map{W}{M}$ is $\pi_1$-injective.
    To do so, we appeal to \cref{lemma:boundary-inclusion-pi1-injective}, 
    which says components of $\bd W \sim \bd W'$ include $\pi_1$-injectively in both $W$ and $W'$. 
    It follows that $\pi_1(M)$ splits as a graph of groups whose vertices are $\pi_1(W)$ and $\pi_1(W')$, 
    and whose edges correspond to the fundamental groups of boundary components.
    As such, the standard tools of Bass--Serre theory tell us the map $\map{\pi_1(W)}{\pi_1(M)}$ is an embedding, 
    and we have our result.
\end{proof}

\subsubsection{Finding a compatible homotopy equivalence}
We now set about proving \cref{lemma:coupling-possible}, which allows us to apply the coupling construction of 
\cref{sec:coupling} in the context of any end-periodic mapping class.

\begin{lemma}
    \label{lemma:coupling-possible}
    If $\Gamma^*$ is an infinite graph with finitely many ends, and $\zeta \in \Maps(\Gamma^*)$ is an end-periodic mapping class,
    then there exist compatible end-periodic homotopy equivalences $\map[g, g']{\Gamma}$ such that $g \in \zeta$, and $\Gamma$ is a graph with no valence-1 vertices. 
\end{lemma}

The lemma is proved in two steps.
Suppose initially that $\map[g]{\Gamma}$ is any end-periodic map whose mapping class is $\zeta$.
Since it is a homotopy equivalence, we may construct as in \cref{thm:compactified-mapping-torus} an end-periodic homotopy inverse $g'$.
We observe this function is almost compatible with $g$, but has one major shortcoming.

\begin{proposition}
    Let $\map[g]{\Gamma}$ be an end-periodic homotopy equivalence with inverse $\map[g']{\Gamma'}$ constructed as in \cref{thm:homotopy-inverse}.
    There exists a coorientation reversing isomorphism $\map[\varphi]{\bd_\pm W_g}{\bd_\mp W_{g'}}$ which preserves ideal joining edges but reverses derived orientations 
    relative to a natural decoration of these boundaries.
\end{proposition}

\begin{proof}
    By \cref{thm:homotopy-inverse}, we can find a well-chosen core $\Gamma_0$ for $g$ such that $g'$ restricts to $g^\inv$ on $\Gamma - \Gamma_0$.
    It follows that $\Gamma_0$ is also a core for $g'$, and writing $\seq[n \in \ZZ]{B_n}$ and $\seq[n \in \ZZ]{B'_n}$ for the 
    block decompositions associated to $\Gamma_0$ induced by $g$ and $g'$ respectively, we observe that $B_n = B'_{-n}$ for all $n \in \ZZ$.
    Letting $W = W_g$ and $W' = W_{g'}$, we obtain by \cref{prop:boundary-structure} a homeomorphism $\map[\varphi]{\bd_\pm W}{\bd_\mp W'}$ 
    mapping $\psi_g^\infty(e) \mapsto \psi_{g'}^\infty(e)$ for each edge $e$ of $B_{1} \cup B_{-1}$.

    Not only does $\varphi$ preserve the ideal joining edges of each boundary component, 
    it can be arranged to respect the subdivisions of these edges in the decorated boundary graphs.
    Indeed, as $e$ has equal period $q$ under both $g$ and $g'$, we see that $\psi_g^\infty(e)$ and $\psi_{g'}^\infty(e)$ 
    represent joining edges of their respective boundaries which are subdivided exactly $q$ times.
    One can compose $\varphi$ with an isotopy supported on $\psi_g^\infty(e)$ to ensure such subdividing vertices align.
    By \cref{rmk:derived-orientations}, the derived orientation on $\psi_g^\infty(e)$ as an edge of $\bd_\pm W$
    must be opposite the derived orientation $\psi_{g'}^\infty(e)$ obtains as an edge of $\bd_\mp W'$, 
    so $\varphi$ reverses decorated edges.
\end{proof}

As things stand, $g$ and $g'$ are not necessarily compatible, because $\varphi$ is not decoration preserving. 
However, one way forward would be to find a component-wise automorphism $\map[\sigma]{\bd W}$ that reverses the orientation of every ideal joining edge,
since the composition $\varphi \sigma$ would then respect derived orientations.
Since a priori there is no reason to expect $\bd W$ should posses such an automorphism, ensuring its existence requires some control over the structure of boundary components.
The observation that an edge-reversing homeomorphism always exists for a graph with only one vertex motivates our next definition.

\begin{definition}[Boundary-collapsed maps]
    We call the end-periodic map $\map[g]{\Gamma}$ \term{boundary-collapsed} if each component of $\bd W$ has exactly one vertex.
\end{definition}

By \cref{prop:boundary-structure}, it is clear that $g$ is boundary-collapsed if and only if it induces a block decomposition 
of $\Gamma$ wherein each nonzero block has one vertex per component.
Our claim is that every end-periodic mapping class contains a boundary-collapsed representative.

\begin{lemma}
    \label{lemma:collapsed-map-exists}
    Let $\Gamma$ be an infinite connected graph and $\map[g]{\Gamma}$ an end-periodic graph map.
    Then there exists an end-periodic, boundary-collapsed graph map $\map[g^\dag]{\Gamma^\dag}$ such that $\Gamma$ is proper homotopy equivalent to $\Gamma^\dag$, and the diagram
    \begin{center} 
        \includegraphics[]{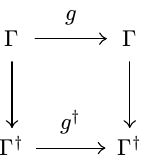}       
    \end{center}
    commutes up to homotopy.
    Furthermore, $\Gamma^\dag$ can be arranged to have no valence-1 vertices.
\end{lemma}

\begin{proof}
    Begin by fixing a core $\Gamma_0$ for $g$ and using it to construct end-invariant maximal tree $T$ as in \cref{lemma:maximal-tree}.
    We define $\Gamma'$ to be the quotient of $\Gamma$ obtained by collapsing each component of the forest $T \cap \sg{B_n}$ to a vertex for all $n \neq 0$.
    Since collapsed regions represent compact, simply connected, and pairwise disjoint subgraphs, the natural map $\map[\pi]{\Gamma}{\Gamma'}$ is both proper and a deformation retraction, so that the graph $\Gamma'$ has the same homotopy type and number of ends as the original.

    The end-invariance of $T$ ensures that $g(T \cap \sg{B_n}) = T \cap \sg{B_{n + 1}}$ for all $n \neq 0$, so the descent of $g$ to $\Gamma'$ is automatically well-defined everywhere except $\pi(\sg{B_{-1}})$.
    We can thus obtain an end-periodic map $g'$ on $\Gamma'$ by first homotoping $g$ to be constant on each component $C$ of $\sg{B_{-1}}$ before passing to the quotient.
    Choose a base vertex $v \in C$, and for every edge $e \in \edges{C - T}$, let $\gamma_e$ denote the loop $\path{v, \bd_0e}_T \cdot e \cdot \path{\bd_1e, v}_T$.
    Since $g$ maps $\gamma_e$ to a loop in $B_0$ based at $g(v)$, it is homotopic rel $\Gamma - C$ to a function $g_1$ that sends $T \cap C$ to $g(v)$ while taking each $e \in \edges{C - T}$ to $g(\gamma_e)$.
    We observe that $g_1$ is still end-periodic, with a well-chosen core now given by $\Gamma_0 \cup B_{-1}$ (the augmentation is necessary since $g_1$ need not be a homeomorphism of $B_{-1}$).
    Furthermore, $g_1$ descends to a map $\map[g']{\Gamma'}$, which must likewise be end-periodic with core $\Gamma'_0 \coloneq \pi(\Gamma_0 \cup B_{-1})$.
    Each nonzero block associated to $\Gamma'_0$ is the quotient of a nonzero block $B_n$ of $\Gamma$, and since $\pi(B_n)$ has exactly one vertex per component, it follows that $g'$ is boundary-collapsed.

    If $\Gamma'$ has no valence-1 vertices, then $g'$ is the map we are seeking.
    Otherwise, it is clear that all valence-1 vertices lie in the finite subgraph $\Gamma'_0$.
    These can be removed by iterating the following procedure.

    Supposing $v \in \Gamma'$ is a valence-1 vertex with incident edge $e$, let $\Gamma^\dag$ be the subgraph of $\Gamma'$ obtained by removing $e \cup \{v\}$, and $\map[\pi']{\Gamma'}{\Gamma^\dag}$ the deformation retraction which collapses $\cl{e}$ to the vertex opposite $v$.
    We note that the map $g^\dag$ given by the restriction of $\pi' g'$ to $\Gamma^\dag$ is end-periodic because it matches $g'$ on all but finitely many edges.
    In particular the finite subgraph $\Gamma_0^\dag \coloneq \pi'(\Gamma'_0)$ represents a core for $g^\dag$: if $\cl{\Gamma'_\pm}$ denotes the closure of the positive/negative nesting domain associated to $\Gamma'_0$, then the subgraphs $g'(\cl{\Gamma'_-})$ and $(g')^{-1}(\cl{\Gamma'_+})$, which have no valence-1 vertices by \cref{def:wellchosen.homeo}, cannot contain $v$.
    This means $(g^\dag)^{\pm 1}$ restricts to $(g')^{\pm 1}$ on $\cl{\Gamma'_\mp}$, so that $\Gamma'_+$ and $\Gamma'_-$ are nesting domains for $g^\dag$ whose union coincides with $\Gamma^\dag - \Gamma_0^\dag$.
    Consequently, $g^\dag$ is boundary-collapsed.
    Because it also forms a square diagram with $g'$ that commutes up to homotopy, we conclude $g^\dag$ satisfies all requirements of the lemma.
    Furthermore, since any remaining valence-1 vertices of $\Gamma^\dag$ are contained in $\Gamma_0^\dag$, 
    whose vertex count is one less than $\nverts{\Gamma'_0}$, this removal process eventually terminates with a map whose domain is free of valence-1 vertices.
\end{proof}

Everything is now in place to prove \cref{lemma:coupling-possible}.

\begin{proof}[Proof of \Cref{lemma:coupling-possible}]
    Let $\map[g]{\Gamma}$ be an end-periodic representative of $\zeta$.
    By \cref{lemma:collapsed-map-exists}, we may assume $g$ is boundary-collapsed, and that its domain has no valence-1 vertices.
    \Cref{thm:homotopy-inverse} gives an end-periodic map $\map[g']{\Gamma}$ whose
    compactified mapping torus $W'$ comes with a homeomorphism $\map[\varphi]{\bd_\pm W}{\bd_\mp W'}$
    that respects subdivisions of ideal joining edges but reverses their derived orientations.
    Since $g$ is boundary-collapsed, every component of the graph $\bd W$ has one vertex, and we can define a homeomorphism
    $\map[\sigma]{\bd W}$ that acts on vertices and subgraph edges by the identity while sending each joining edge to its inverse.
    The composition $\varphi \sigma$ is then isotopic to a decoration preserving map, so $g$ and $g'$ are compatible.
\end{proof}

\subsubsection{Finding an appropriate subgraph $\Theta$}
Our next task is to prove the following.

\begin{lemma}
    \label{lemma:f-homotopy-equivalence}
    Suppose $\map[g]{\Gamma}$ and $\map[g']{\Gamma'}$ are
    $h$-compatible end-periodic homotopy equivalences of graphs with no valence 1 vertices.
    Then there exist cellulations of $W = W_g$ and $W' = W_{g'}$ such that the principal subgraph
    $\Theta$ of their $h$-couple has a first return map which is a homotopy equivalence.
\end{lemma}

Since $\Theta$ is a finite graph, our main tool will be Stallings folding \cite{Stallings:finiteGraphs}.
Recall that a \term{fold} is a graph map given by the natural quotient $\map[f]{\Gamma}{\Gamma / (\cl{e_1} \sim \cl{e_2})}$ identifying a pair of closed, oriented edges $e_1 \neq e_2$ with the same initial vertex $\bd_0 e_1 = \bd_0 e_2$.
Note that $f$ is a homotopy equivalence if and only if $\bd_1 e_1 \neq \bd_1 e_2$, in which case it is called a \term{type-1 fold}.
Otherwise, $f$ is a \term{type-2 fold}, which clearly fails to be $\pi_1$-injective.

Given arbitrary graphs $\gX, \gY$ and a graph map $\map[f]{\gX}{\gY}$ which does not collapse edges, 
we let $\gX^*$ denote the subdivision of $\gX$ that makes $\map[f]{\gX^*}{\gY}$ combinatorial.
If $\gX^*$ is finite, one can always write $f$ as a finite composition of combinatorial graph maps
\[
    \gX^* = \gX_1 \xrightarrow{f_1} \gX_2 \xrightarrow{f_2} \cdots \space \gX_k \xrightarrow{f_k} \gY
\]
where each $f_1, \dots, f_{k - 1}$ is a fold, and the final map $f_k$ is a unique immersion \cite{Stallings:finiteGraphs}.
When $f$ is a homotopy equivalence and $\gX$ has no valence-1 vertices, it is easy to show such an $f_k$ must also be a homeomorphism.
The first step toward proving \cref{lemma:f-homotopy-equivalence} will be demonstrating a similar fact for end-periodic graph maps, namely:

\begin{proposition}
    \label{lemma:folds-to-homeo}
    Every end-periodic graph map $\map[g]{\Gamma}$ can be written as a finite composition
    \[
        \Gamma^* = \Gam_1 \xrightarrow{g_1} \Gam_2 \xrightarrow{g_2} \cdots \space \Gam_k \xrightarrow{g_k} \Gamma
    \]
    where the first $k - 1$ maps are folds, and the final map is an immersion.
    Furthermore, if $g$ is a homotopy equivalence, and $\Gamma$ has no valence-1 vertices, then $g_k$ is a homeomorphism.
\end{proposition}

\begin{proof}
    Since $g$ is end-periodic, it restricts to an isomorphism everywhere outside some compact core $C$ 
    (we temporarily put aside the notation $\Gamma_0$ avoid confusing folded graphs $\Gamma_n$ with enlargements).
    Thus, $\map[\restr{g}{C} = f]{C^*}{g(C)}$ is a combinatorial map of finite graphs which factors as the composition
    \[
        C^* = C_1 \xrightarrow{f_1}  C_2 \xrightarrow{f_2} \cdots \: C_k \xrightarrow{f_k} g(C),
    \]
    where $f_k$ is an immersion and the others are type-1 folds.
    Extend $f_1$ over $\Gamma^* \eqcolon \Gam_1$ by the identity to define $\map[g_1]{\Gam_1}{\Gam_2}$, where $\Gam_2 \coloneq g_1(\Gam_1)$.
    Continuing inductively, one obtains a sequence $\Gam_1 \xrightarrow{g_1} \Gam_2 \xrightarrow{g_2} \cdots \space \Gam_k$ 
    in which each $g_n$ is the extension of $f_n$ by the identity and therefore a fold.
    Because $\Gam_n - C_n = \Gamma - C$ for all $n$, and $g$ restricted to $\Gamma - C$ is a homeomorphism onto its image,
    the extension of $f_k$ to $\Gam_k$ by $g$, which we denote $\map[g_k]{\Gam_k}{\Gamma}$, must be an immersion satisfying $g = g_k \cdots g_1$.

    Suppose now that $\map[g]{\Gamma}$ is a homotopy equivalence and $\Gamma$ has no valence-1 vertices.
    Our claim is that under such circumstances, $g_k$ must be a homeomorphism.

    We start by proving $g_k$ is a surjection, which amounts to showing $g$ is onto since the images $g_k(\Gam_k) = g(\Gamma)$ coincide.
    By continuity, it is enough to verify $g(\Gamma) \subeq \Gamma - \zskele\Gamma$.
    Every open edge $e \in \edges{\Gamma}$ either lies on a cyclically reduced loop (that is, an immersed circle) in $\Gamma$, 
    or its removal results in the separation of $\Gamma$ into two connected subgraphs, at least one of which is acyclic.
    In the first case, the image of $g$ must contain $e$ because $g$ is a homotopy equivalence.
    Otherwise, any acyclic component $A$ of $\Gamma - e$ contains a neighborhood of least one end, because $\Gamma$ has no valence-1 vertices by assumption.
    As $g$ acts on $\Ends(\Gamma)$ by permutation, $g(\Gamma) \cap A$ must be nonempty, 
    and seeing that $g(\Gamma)$ also intersects cyclic components by the first case, 
    it follows that the image of $g$ meets both components of $\Gamma - e$.
    Thus, $e \subeq g(\Gamma)$ by the continuity of $g$ and connectedness of $\Gamma$.

    Now, if $g_k$ failed to be injective, we could find distinct vertices $u, v \in \verts{\Gam_k}$ such that $g_k(u) = g_k(v) = w \in \verts{\Gamma}$, as well as a reduced edge path $\rho = e_1 \mathbin{\cdots} e_n$ in $\Gam_k$ between them. 
    The composition $g_k \rho$ is a loop in $\Gamma$ based at $w$, which can be written $g_k(e_1) \mathbin{\cdots} g_k(e_n)$ since $g_k$ is combinatorial.
    As $g_k$ is an immersion, the path $g_k \rho$ is also reduced, and in particular cannot be null-homotopic.
    Thus, $1 \neq [g_k \rho] \in \pi_1(\Gamma, w)$, and since $g_k$ is a homotopy equivalence, there exists $[\gamma] \in \pi_1(\Gam_k, u)$ such that $(g_k)_*([\gamma]) = [g_k \rho]^{-1}$.
    Writing $\gamma$ as a reduced edge path $d_1 \mathbin{\cdots} d_m$ with $\bd_0 d_1 = \bd_1 d_m = u$, the composition $g_k \gamma = g_k(d_1) \mathbin{\cdots} g_k(d_r)$ is then a reduced representative of $(g_k)_*([\gamma])$ satisfying $[g_k \rho] [g_k \gamma] = 1$.
    It follows that the concatenation
    \[
        g_k \rho \cdot g_k \gamma = g_k(e_1) \mathbin{\cdots} g_k(e_n) \cdot g_k(d_1) \mathbin{\cdots} g_k(d_m)
    \]
    can be homotoped to a constant path by successively removing subpaths of the form $e \cdot e^{-1}$.
    Because $g_k \rho$ and $g_k \gamma$ are both reduced, all cancellations must occur at the interface between them, implying that $m = n$, and $g(e_{n - i}) = g(d_i)^{-1} = g(d_i^{-1})$ for all $i = 1, \dots, n$.
    But since $\gamma$ is a loop while $\rho$ is not, there must also be edges $e_{n - i} \neq d_i^{-1}$, which share a terminal vertex when $i$ is minimal.
    These map to the same edge under $g_k$, contradicting the fact that it is an immersion.
\end{proof}

We pause for a moment to observe a corollary of the previous proposition, 
which points out a structural characteristic of mapping tori of end-periodic homotopy equivalences.

\begin{corollary}
    \label{cor:euler-characteristic-preserved}
    If $\map[g]{\Gamma}$ is an end-periodic homotopy equivalence, 
    then the positive and negative boundaries of $W_g$ have the same Euler characteristic: $\chi(\bd_+ W_g) = \chi(\bd_- W_g)$ .
\end{corollary}

\begin{proof}
    Fix a proper core $\Gamma_0$ for $g$.
    By \cref{prop:boundary-structure}, $\chi(\bd_\pm W_g) = \nverts{B_{\pm 1}} - \nedges{B_{\pm 1}}$, so we observe
    \begin{align*}
        \chi(\Gamma_0 \cup B_{\pm 1}) & = \nverts{\Gamma_0 \cup B_{\pm 1}} - \nedges{\Gamma_0 \cup B_{\pm 1}}             \\
                                      & = \nverts{\Gamma_0} + \nverts{B_{\pm 1}} - \nedges{\Gamma_0} - \nedges{B_{\pm 1}} \\
                                      & = \chi(\Gamma_0) + \chi(\bd_\pm W_g).
    \end{align*}
    Suppose initially that $\Gamma$ has no valence-1 vertices.
    Then, since $g$ is a homotopy equivalence, it factors as $g = g_k \cdots g_1$ where $g_1, \dots, g_{k - 1}$ are type-1 folds and $g_k$ is a homeomorphism.
    All of these maps preserve Euler characteristic, so $\chi(K) = \chi(g(K))$ for any finite subgraph $K \subeq \Gamma$.
    Since $\Gamma_0$ is a proper core, we have $g(\Gamma_0 \cup B_{-1}) = \Gamma_0 \cup B_1$.
    This implies $\chi(\Gamma_0 \cup B_{-1}) = \chi(\Gamma_0 \cup B_1)$, which yields the desired equality when combined with the equation above.

    Now, if $\Gamma$ does have valence-1 vertices, we can instead consider the end-periodic map $\map[g^\dag]{\Gamma^\dag}$ that results from collapsing $g$ as in \cref{lemma:collapsed-map-exists}.
    Note that this construction allows us to assume $\Gamma^\dag$ is free of valence-1 vertices, so we have $\chi(\bd_+ W_{g^\dag}) = \chi(\bd_- W_{g^\dag})$ by the previous case.
    Since each component of $\bd_\pm W_{g^\dag}$ corresponds to a component of $\bd_\pm W_g$ with a maximal subtree collapsed, we see at once that $\chi(\bd_\pm W_g) = \chi(\bd_\pm W_{g^\dag})$, and the conclusion follows.
\end{proof}

With \cref{lemma:folds-to-homeo} established, we can give the proof of \cref{lemma:f-homotopy-equivalence}.

\begin{proof}[Proof of \Cref{lemma:f-homotopy-equivalence}]
    Let $\Gamma_0$ and $\Gamma'_0$ be cores for $g$ and $g'$ inducing $h$-compatible cellulations of $W$ and $W'$.
    By \cref{prop:big-core}, we can enlarge $\Gamma_0$ if necessary to ensure it is a proper core for $g$.
    Since the positive and negative blocks of the decomposition induced by this enlargement are isomorphic to the blocks $B_{\pm 1}$ associated to the original core, 
    it fixes the same decoration of $\bd W$, so that $h$ remains decoration preserving with respect to this new cellulation of $W$.
    Since the decoration is also invariant under choice of boundary neighborhoods, we can require that $W$ be cellulated with a cutoff of at least 2 everywhere.
    Thus, by \cref{prop:rebase-interior-subgraph}, we may assume with no loss of generality that the interior subgraph $\Delta$ of $W$ is exactly $\Gamma_2 = \bigcup_{i = - 2}^2 B_i$.
    In the same way, we assume $\Gamma'_0$ is a proper core for $g'$ such that $\Delta' = \Gamma'_2 = \bigcup_{i = - 2}^2 B'_i$.
    The resulting $h$-couple $M(W, W', h)$ has a principal subgraph $\Theta = \Delta \cup \Delta' \cup S_+ \cup S_-$ whose structure 
    is similar to the schematic in \cref{fig:Theta-schematic}.
    \begin{figure}[htpb]
        \centering
        \includegraphics[]{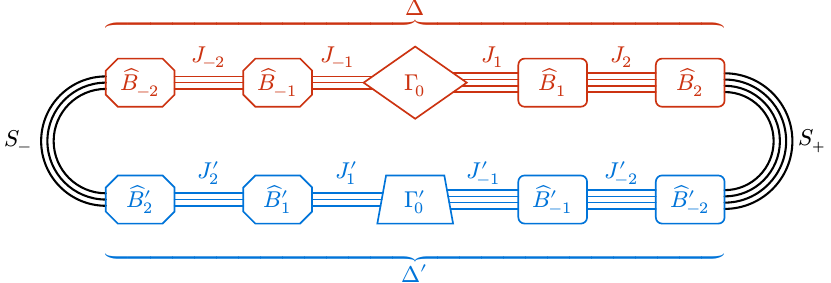}
        \caption{
            Schematic depiction of $\Theta$. 
            This cartoon depicts a scenario where all ends of $\Gamma$ have period 1, 
            though in general $S_\pm$ may also contain edges between vertices of $\bd_\pm \Gamma_0 \cup \widehat{B}_{\pm 1}$ 
            and $\bd_\mp \Gamma'_0 \cup \widehat{B}'_{\mp 1}$.
        }
        \label{fig:Theta-schematic},
    \end{figure}

    Letting $f$ denote the first return map of $\Theta$ under the extended semiflow, 
    \cref{prop:properties-of-Theta.f-matches-g} shows that $\restr{f}{\Delta - B_2} = \restr{g}{\Delta - B_2}$.
    Since $\Gamma_0$ is a proper core for $g$, it follows that $g(B_{-1} \cup \Gamma_0 \cup B_1) = \Gamma_0 \cup B_1 \cup B_2$. 
    Furthermore, as an end-periodic homotopy equivalence of a graph with no valence-1 vertices, \cref{lemma:folds-to-homeo} implies $g = g_k g_{k - 1} \dots g_1$, 
    where $g_k$ is a homeomorphism onto $\Gamma$, and the rest are all type-1 folds.
    Setting $G_i \coloneq g_{i-1} \cdots g_1$ with $G_0 \coloneq \id_\Gamma$ allows us to express the domain and codomain of these functions as $\map[g_i]{G_{i}(\Gamma)}{G_{i+1}(\Gamma)}$.
    Now, because $g$ is a homeomorphism everywhere outside $\Gamma_0$, and this core contains the image of every edge mapped non-combinatorially, 
    it follows by induction that every fold $g_i$, $1 \le i < k$ must identify two edges in $G_i(\Gamma_0)$.
    Thus, we obtain a folding sequence for $f$ which mirrors that of $g$ by defining functions
    $\map[f_i]{F_i(\Theta)}{F_{i+1}(\Theta)}$ (with $F_0 \coloneq \id_\Theta, F_i \coloneq f_{i - 1} \cdots f_1$), 
    which match $g_i$ on $F_i(\Gamma_0)$, and restrict to the identity elsewhere on $F_i(\Theta)$.
    This gives $f = f_k f_{k - 1} \cdots f_1$, where the first $k - 1$ functions are type-1 folds, 
    and the restriction of $f_k$ to $F_k(B_{-1} \cup \Gamma_0 \cup B_{1})$ is a homeomorphism onto $\Gamma_0 \cup B_1 \cup B_2$.

    Noting that $\restr{f}{\Delta' - B'_2} = \restr{g'}{\Delta' - B'_2}$, 
    the same procedure also allows us to express $f$ as $f = f'_{k'} f'_{k' - 1} \cdots f'_1$, 
    where each $f'_i$, $1 \le i < k'$ is a type-1 fold of two edges in $F'_i(\Gamma'_0)$, and $f'_{k'}$ restricts to a homeomorphism 
    $F'_{k'}(B'_{-1} \cup \Gamma'_0 \cup B'_1) \to \Gamma'_0 \cup B'_1 \cup B'_2$.
    Since the folded regions are disjoint, we can perform folds $f'_1$ through $f'_{k' - 1}$ immediately after $f_1$ through $f_{k - 1}$ in order to factor $f$ as
    \[
        f = f^* F'_{k'} F_k = f^* f_\ell^* \cdots f_1^*, 
        \quad \text{where} 
        \quad \ell = k + k',
        \quad f_i^* = 
        \begin{cases}
            f_i             & \text{if } 1 \le i < k \\
            f'_{i - k + 1}  & \text{if } k \le i < \ell.
        \end{cases}
    \]
    The map $f^*$ acts as $f_k$ on 
    $f_\ell^* \cdots f_1^*(B_{-1} \cup \Gamma_0 \cup B_{1})$ and as $f'_{k'}$ on $f_\ell^* \cdots f_1^*(B'_{-1} \cup \Gamma'_0 \cup B'_1)$, 
    which is to say a homeomorphism onto $\Gamma_0 \cup B_1 \cup B_2$ and $\Gamma'_0 \cup B'_1 \cup B'_2$ respectively.
    On the subgraphs $\Pi_+ \coloneq \cl{B_2 \cup S_+ \cup B'_{-2}}$ and $\Pi_- \coloneq \cl{B'_2 \cup S_- \cup B_{-2}}$, 
    which are unaffected by the initial folds, 
    $f^*$ restricts to $f$ itself.
    Our last order of business will be to demonstrate that $\Pi_+$ and $\Pi_-$ are respectively 
    isomorphic to $\cl{S_+ \cup B'_{-2} \cup B'_{-1}}$ and $\cl{S_- \cup B_{-2} \cup B_{-1}}$ under $f$.
    This claim, once established, shows $f^*$ must be a homeomorphism onto $\Theta$.
    Thus, all functions in the decomposition of $f$ induce isomorphisms of $\pi_1$, proving the lemma.

    The proof that $f$ is an isomorphism between $\Pi_1$ and the closure of $S_+ \cup B'_{-2} \cup B'_{-1}$ 
    is in fact an easy consequence of \cref{prop:properties-of-Theta}.
    By \ref{prop:properties-of-Theta.f-matches-g}, we already know that $f$ restricts to $g'$ on $\cl{B'_{-2}}$, and is therefore an isomorphism onto $\cl{B'_{-1}}$ by end-periodicity.
    At the same time, $f$ is an isomorphism $\cl{B_2 \cup S_+} \to \cl{B'_{-2} \cup S_+}$ by \ref{prop:properties-of-Theta.isomorphisms-of-lambda}.
    Together, these imply $\restr{f}{\Pi_+}$ is an isomorphism with the desired image.
    The same holds for $\restr{f}{\Pi_-}$ by direct analogy, so we have our result.
\end{proof}

\subsubsection{\texorpdfstring{$\pi_1$}{pi\_1}-injectivity of the boundary}

Our final lemma is straightforward. We must show:

\begin{lemma}
    \label{lemma:boundary-inclusion-pi1-injective}
    If the end-periodic map \(\map[g]{\Gamma}\) induces an injection on \(\pi_1(\Gamma)\), 
    then the inclusion \(\Sigma_E \to W\) of each boundary component \(\Sigma_E \subeq \bd_\pm W\) is also \(\pi_1\)-injective.
\end{lemma}

\begin{proof}
    Let $\Sigma_E$ be a component of the positive boundary $\bd_+ W$ (the proof we give has a natural analog for components of $\bd_- W$).
    Fix a cellulation of $W$, and let $\Lambda$ denote the corresponding principal subgraph.
    We will make the simplifying assumption that $\Sigma_E$ is a rose with vertex $v_\infty$ and edges $e_\infty^1, \dots, e_\infty^k$ in standard orientation.
    This streamlines notation considerably, since it allows us to view each edge as a generator of $\pi_1(\Sigma_E)$.
    The general argument can be obtained by fixing an end-invariant maximal tree 
    $T$ and considering a generating set for $\pi_1(\Sigma_E)$ corresponding to edges of $\Sigma_E - (\sg{B_1} \cap T)$.

    Since the nicely embedded graph $\Lambda$ has an oriented $[0, 1]$-bundle neighborhood in $W$ determined by the semiflow, 
    it is dual to a cohomology class $\lambda \in H^1(W; \ZZ)$ which evaluates elements of $H_1(W)$ 
    by counting the signed intersections of $\Lambda$ with transverse representative 1-cycles.
    Since $\Lambda \cap \Sigma_E$ consists of exactly $\pd{E}$ points along each ideal joining edge, we see
    \[
        \lambda([e_\infty^i]) = \begin{cases}
            0   & \text{if } e_\infty^i \text{ is a subgraph edge} \\
            - \pd{E} & \text{if } e_\infty^i \text{ is a joining edge}.
        \end{cases}
    \]
    Note that the minus sign arises because $\Sigma_E$ lies in the positive boundary, per \cref{rmk:derived-orientations}.

    The preimages of $v_\infty$ and $e_\infty^i$ under $\psigoo$ are sets of disjoint vertices and edges in 
    $\Gamma_+$ which we denote by $\seq[n \ge 1]{v_n}$ and $\seq[n \ge 1]{e_n^i}$ respectively.
    Since $\psi_g^\infty$ induces homeomorphisms $\St(v_n, \Gamma_+) \homeo \St(v_\infty, \Sigma_E)$ and $e_n^i \homeo e_\infty^i$,
    we see that for each $n \ge 1$ there is a unique lift of the ideal edge $e_\infty^i$ to an edge $e_m^i$ of $\Gamma_+$ beginning at $v_n$.
    Provided that $n > \pd{E}$, this lift has terminal vertex $v_n$ if $e_\infty^i$ is a subgraph edge, 
    $v_{n + \pd{E}}$ if $e_\infty^i$ is a joining edge in standard orientation, and $v_{n - \pd{E}}$ otherwise.
    Since this terminal vertex represents a basepoint for lifting subsequent edges, 
    it follows by induction that any closed edge loop $\gamma$ in $\Sigma_E$ can be lifted to an edge path $\wt{\gamma}$ in $\Gamma_+$ 
    beginning at $v_n$ and ending at $v_{n - \lambda([\gamma])}$ for all $n > \lambda([\gamma])$.

    With this established, $\pi_1$-injectivity of $\Sigma_E \to W$ easily follows.
    Let $\gamma$ be a nontrivial loop in $\Sigma_E$.
    If $\lambda([\gamma]) \neq 0$, then $\gamma$ must be essential in $W$, since its homotopy class lies outside the kernel of $\pi_1(W) \to H_1(W) \xrightarrow{\lambda} \ZZ$.
    Otherwise, $\gamma$ is homotopic to a nontrivial loop $\wt{\gamma}$ in $\Gamma$, and the $\pi_1$-injectivity of $g$ implies $g_*^k ([\, \wt{\gamma} \,]) \neq 1$ for all positive $k$.
    Since relations in $\pi_1(W)$ are of the form $t \alpha t^{-1} = g_*(\alpha)$ for $\alpha \in \pi_1(\Gamma)$, 
    this implies $\wt{\gamma}$ is essential in the mapping torus.
\end{proof}

\bibliographystyle{amsalpha}
\bibliography{refs}

\end{document}